\documentclass{amsart}
\usepackage{verbatim}
\usepackage[textsize=scriptsize]{todonotes}
\usepackage{etoolbox}
\newtoggle{draft}
\newtoggle{final}

\toggletrue{final}

\usepackage{etoolbox}
\iftoggle{draft} {
\usepackage[margin=1.5in]{geometry}
}{ 
\usepackage[margin=1in]{geometry}
\setlength{\marginparwidth}{0.75in}
}
\def\PSh{\mathcal{S}\mathrm{hv}}

\usepackage{amsmath}
\usepackage{amssymb}
\usepackage{amsthm}
\usepackage{amscd}
\usepackage{enumerate}
\usepackage[pdfusetitle,unicode,hidelinks]{hyperref}
\usepackage{bbm}
\usepackage{etoolbox}

\usepackage[utf8]{inputenc}
\usepackage[T1]{fontenc}

\newcommand{\tomemail}{\href{mailto:tom.bachmann@zoho.com}{tom.bachmann@zoho.com}}

\newtheorem{proposition}{Proposition}
\newtheorem{corollary}[proposition]{Corollary}
\newtheorem{lemma}[proposition]{Lemma}
\newtheorem{theorem}[proposition]{Theorem}

\newtheorem*{conjecture*}{Conjecture}
\newtheorem*{theorem*}{Theorem}
\newtheorem*{corollary*}{Corollary}
\newtheorem*{proposition*}{Proposition}
\newtheorem*{lemma*}{Lemma}
\theoremstyle{definition}
\newtheorem{definition}[proposition]{Definition}

\newtheorem*{definition*}{Definition}
\newtheorem*{construction*}{Construction}
\theoremstyle{remark}
\newtheorem{remark}[proposition]{Remark}
\newtheorem*{remark*}{Remark}

\newtheorem{example}[proposition]{Example}
\newtheorem*{example*}{Example}

\newcommand{\id}{\operatorname{id}}
\newcommand{\Z}{\mathbb{Z}}

\newcommand{\Q}{\mathbb{Q}}

\let\scr=\mathcal
\let\bb=\mathbb
\newcommand{\Gm}{{\mathbb{G}_m}}
\newcommand{\Gmp}[1]{{\mathbb{G}_m^{\wedge #1}}}
\def\A{\bb A}
\def\P{\bb P}
\def\R{\bb R}
\newcommand{\1}{\mathbbm{1}}

\newcommand{\DM}{\mathcal{DM}}
\newcommand{\SH}{\mathcal{SH}}

\DeclareMathOperator*{\colim}{colim}

\let\lim=\relax
\DeclareMathOperator*{\lim}{lim}
\def\Map{\mathrm{Map}}

\def\map{\mathrm{map}}

\def\PSh{\mathcal{P}}

\def\Cat{\mathcal{C}\mathrm{at}{}}

\def\Fun{\mathrm{Fun}}

\newcommand{\wequi}{\simeq}

\def\adj{\leftrightarrows}

\DeclareRobustCommand{\ul}{\underline}
\newcommand{\heart}{\heartsuit}

\newcommand{\Hom}{\operatorname{Hom}}

\def\op{\mathrm{op}}

\let\cat=\mathrm
\def\Sm{{\cat{S}\mathrm{m}}}
\def\Sch{\cat{S}\mathrm{ch}{}}

\def\Zar{\mathrm{Zar}}

\newcommand{\et}{{\acute{e}t}}

\newcommand{\lra}[1]{\langle #1 \rangle}
\def\ph{\mathord-}

\numberwithin{proposition}{section}

\iftoggle{draft} {
\newcommand{\NB}[1]{\todo[color=gray!40]{#1}}
}{ 
\newcommand{\NB}[1]{}
}
\iftoggle{final} {
\renewcommand{\todo}[1]{}
}{ 
\newcommand{\TODO}[1]{\todo[color=red]{#1}}
}

\newcommand{\proet}{{pro{\acute{e}t}}}
\newcommand{\hyp}{\wedge}
\newcommand{\comp}{\wedge}
\newcommand{\Shv}{\mathcal{S}\mathrm{hv}}

\def\imap{\underline{\mathrm{map}}}
\newcommand{\pcd}{\operatorname{pcd}}
\newcommand{\cd}{\operatorname{cd}}

\input cyracc.def
\DeclareFontFamily{U}{russian}{}
\DeclareFontShape{U}{russian}{m}{n}
        { <5><6> wncyr5
        <7><8><9> wncyr7
        <10><10.95><12><14.4><17.28><20.74><24.88> wncyr10 }{}
\DeclareSymbolFont{Russian}{U}{russian}{m}{n}
\DeclareSymbolFontAlphabet{\mathcyr}{Russian}
\makeatletter
\let\@math@cyr\mathcyr
\renewcommand{\mathcyr}[1]{\@math@cyr{\cyracc #1}}
\makeatother

\def\Zptw{\hat\Z_p(1)}
\def\Sptw{\hat\1_p(1)}
\def\nSptw{\hat\1_p(-1)}
\def\H{\mathrm{H}}

\title{Rigidity in étale motivic stable homotopy theory}
\date{\today}

\author{Tom Bachmann}
\address{Department of Mathematics, Massachusetts Institute of Technology,
Cambridge, MA, USA}
\email{\tomemail}

\begin{document}

\maketitle

\begin{abstract}
For a scheme $X$, denote by $\SH(X_\et^\hyp)$ the stabilization of the
hypercompletion of its étale $\infty$-topos, and by $\SH_\et(X)$ the
localization of the stable motivic homotopy category $\SH(X)$ at the
(desuspensions of) étale hypercovers.
For a stable $\infty$-category $\scr C$, write $\scr C_p^\comp$ for
the $p$-completion of $\scr C$.

We prove that under suitable finiteness hypotheses, and assuming that $p$ is
invertible on $X$, the canonical functor \[ e_p^\comp: \SH(X_\et^\hyp)_p^\comp \to
\SH_\et(X)_p^\comp \] is an equivalence of $\infty$-categories. This generalizes
the rigidity theorems of Suslin-Voevodsky \cite{suslin1996singular}, Ayoub
\cite{ayoub2014realisation} and Cisinski-Déglise \cite{cisinski2013etale} to the
setting of spectra. We deduce that under further regularity hypotheses on $X$,
if $S$ is the set of primes not invertible on $X$, then the endomorphisms of the
$S$-local sphere in $\SH_\et(X)$ are given by étale hypercohomology with
coefficients in the $S$-local classical sphere spectrum:
\[ [\1[1/S], \1[1/S]]_{\SH_\et(X)} \wequi \mathbb{H}^0_\et(X, \1[1/S]). \]
This confirms a conjecture of Morel.

The primary novelty of our argument is that we use the pro-étale topology
\cite{bhatt2013pro} to construct directly an invertible object $\Sptw[1] \in
\SH(X_\et^\hyp)_p^\comp$ with the property that $e_p^\comp(\Sptw[1]) \wequi
\Sigma^\infty \Gm \in \SH_\et(X)_p^\comp$.
\end{abstract}

\tableofcontents

\section{Introduction}
Étale motivic cohomology with finite coefficients invertible in the base
coincides with étale cohomology. In more categorical terms, the canonical
functor $D(X_\et, \Z/n) \to \DM_\et(X, \Z/n)$ is an equivalence, provided that
$1/n \in \scr O(X)^\times$. This was proved for the case where $X$ is the
spectrum of a field by Suslin-Voevodsky \cite[Theorem 4.4]{suslin1996singular}
\cite[Proposition 3.3.3]{voevodsky-triang-motives}. Versions of this result over
more general bases were established by Ayoub for motives without transfers
\cite[Theoreme 4.1]{ayoub2014realisation} and by Cisinski-Déglise for motives
with transfers \cite[Theorem 4.5.2]{cisinski2013etale}. It is a natural
question to ask if there is a ``spectral'' version of these results. The main aim of
this article is to establish the following positive answer.

\begin{theorem*}[see Theorem \ref{thm:main}]
Let $X$ be a locally $p$-étale finite scheme and $p$ a prime with $1/p \in
X$.\footnote{We will abuse notation and write $1/p \in X$ instead of $p
\in \scr O(X)^\times$.}
Then the canonical functor
$\SH(X_\et^\hyp)_p^\comp \to \SH_\et(X)_p^\comp$ is an equivalence.
\end{theorem*}

We recall the definitions of the terms in the above theorem in the following set
of remarks.

\begin{remark*}[spectral sheaves and hypercompletion]
We denote by $X_\et$ the small étale site of $X$, i.e. the category of (finitely
presented) étale $X$-schemes, with the Grothendieck topology given by the
jointly surjective quasi-compact families. Associated with this we have the
$\infty$-topos $\Shv(X_\et)$ of sheaves of spaces on $X_\et$, i.e. presheaves of
spaces satisfying étale descent. We denote by $\Shv(X_\et^\hyp)$ its
hypercompletion; in other words these are the presheaves satisfying descent with
respect to all hypercovers. Equivalently, equivalences are detected on homotopy
sheaves \cite{dugger2004hypercovers}. Finally we denote by $\SH(X_\et^\hyp)$ the
stabilization of the $\infty$-topos $\Shv(X_\et^\hyp)$; equivalently this is the
category of spectral hypersheaves: the category of functors from $X_\et^\op$ to
spectra, satisfying descent for all étale hypercovers (or, equivalently, being
local for the family of weak equivalences detected by homotopy sheaves). See
Section \ref{subsec:sheaves-of-spectra} for more about sheaves of spectra.
\end{remark*}

\begin{remark*}[$p$-completion of stable categories]
In spectral settings, there is no evident analog of ``working with $\Z/n$
coefficients''. One standard resolution of this is to use $p$-completion, which
is somewhat analogous to working with $\Z_p$-coefficients, where $\Z_p$ denotes
the ring of $p$-adic integers. Let $\scr C$ be a presentable stable
$\infty$-category (or a triangulated category with small coproducts). For $X, Y
\in \scr C$ the set of homotopy classes of maps from $X$ to $Y$
is naturally an abelian group, and consequently for every $X \in \scr C$ we have
a canonical endomorphism $p = p\id_X = \id_X + \id_X + \dots + \id_X$. We denote
by $X/p$ the cofiber of (cone on) this endomorphism. We call a map $f: X \to Y \in
\scr C$ a $p$-equivalence if $cone(f)/p \wequi 0$, and we denote by $\scr
C_p^\comp$ the localization of $\scr C$ at the $p$-equivalences; under our
assumptions this exists and is in fact equivalent to the full subcategory of
$\scr C$ right orthogonal to all objects $X \in \scr C$ such that $X/p \wequi
0$. See Section \ref{subsec:p-completion} for more about $p$-completion.
\end{remark*}

\begin{remark*}[étale motivic stable homotopy theory]
The motivic stable $\infty$-category $\SH(X)$ is obtained from the
$\infty$-topos $\PSh(\Sm_X)$ by (1) inverting the Čech nerves of Nisnevich coverings, (2)
inverting all maps of the form $\A^1_Y \to Y$ for $Y \in \Sm_X$, and (3)
passing to pointed objects, then stabilizing with respect to the endofunctor
$\wedge \P^1$; see e.g. \cite[Section 5]{morel-trieste} \cite[Sections 2.2 and
4.1]{bachmann-norms}. The category $\SH_\et(X)$ is constructed in exactly the
same way, except that in step (2) we use the étale hypercoverings instead.
Equivalently, $\SH_\et(X)$ is obtained from $\SH(X)$ by inverting the maps of
the form $\Sigma^{\infty-n}_+ \scr Y \to \Sigma^{\infty - n}_+ Y$, for étale
hypercoverings $\scr Y \to Y \in \Sm_X$. The functor of taking the associated locally
constant sheaf $e: \Shv(X_\et) \to \Shv(\Sm_{X,\et})$ induces a functor
$\SH(X_\et^\hyp) \to \SH_\et(X)$, which after $p$-completion induces the
equivalence in the theorem.
\end{remark*}

\begin{remark*}[$p$-étale finiteness]
We call a scheme $X$ \emph{$p$-étale finite} if for every finite type scheme $Y/X$ there
exists $n$ such that every finitely presented, qcqs étale $Y$-scheme $Z$
satisfies $\cd_p(Z) \le n$. We call $X$ \emph{locally $p$-étale finite}
if it admits an étale cover by $p$-étale finite schemes. This holds for example if
$X$ is of finite type over a field of finite virtual $p$-étale cohomological
dimension (this includes all finite fields, separably closed fields, numbers
fields, and $\R$) or $\Z$. We call a scheme (locally) étale finite if it is
(locally) $p$-étale finite for all primes $p$. See Definitions
\ref{def:uniformly-bounded} and \ref{def:etale-finite} and Examples
\ref{ex:uniformly-bounded-dim} and \ref{ex:etale-finite}.
\end{remark*}

With the above theorem at hand, we of course find that $[\1,
\1]_{\SH_\et(X)_p^\comp} \wequi
[\1, \1]_{\SH(X_\et^\hyp)_p^\comp}$. This is significant, since the left hand side is a
priori much more complicated than the right hand side, which is basically
controlled by étale cohomology of $X$ and the classical stable homotopy groups. In
general, one expects to learn essentially everything about a category $\scr C$
by studying $\scr C_p^\comp$ for all $p$, and also the rationalization $\scr
C_\Q$. Since $\SH_\et(X)_\Q \wequi \DM_\et(X, \Q) \wequi \DM(X, \Q)$ is
reasonably well understood, one might hope to patch together all of these
computations to determine $[\1, \1]_{\SH_\et(X)}$; this was the original aim of
the article. It has been fulfilled as follows.

\begin{corollary*}[see Corollary \ref{corr:final}]
Let $X$ be locally étale finite, and $S$ be the set of primes not invertible on $X$.
Assume that $X$ is regular, noetherian and finite dimensional. Then
\[ [\1, \1]_{\SH_\et(X)[1/S]} \wequi \mathbb{H}^0_\et(X, \1[1/S]), \]
where the right hand side denotes étale hypercohomology with coefficients in the
(classical) sphere spectrum (in other words $\pi_0$ of the spectrum of global
sections of the étale hypersheafification of the constant presheaf of spectra
with value the classical sphere spectrum).
\end{corollary*}

\subsection*{Proof strategy (for the main theorem)}
Suppose that the functor $e: \SH(X_\et^\hyp)_p^\comp \to \SH_\et(X)_p^\comp$ is
indeed an equivalence. As a basic sanity check, we should be able to write down
an object $\Sptw[1] \in \SH(X_\et^\hyp)_p^\comp$ such that $e(\Sptw[1]) \wequi
\Gm$. In the abelian situation, say with $\Z/p^n$ coefficients (i.e. in
$\DM_\et(X, \Z/p^n)$), the
corresponding sheaf is $\mu_{p^n}[1]$, the sheaf of $p^n$-th roots of unity. In the
spectral situation however, it is not so obvious (to the author) what the
analogous object is. We know that $\Sptw[1]$ should be an invertible
spectrum, that $\Sptw[1] \wedge \H\Z/p^n \wequi \mu_{p^n}[1]$, and by analogy
with the abelian situation we might guess that $\Sptw[1] \wequi \1[1]$ if the
base has all $p^n$-th roots of unity for all $n$.

It turns out that the construction of $\Sptw[1]$ is central to our proof of the
main theorem, so let us pursue this further.
The last condition gives a clue: even if we don't know how to construct
$\Sptw[1]$ directly, it seems to be a form of $\1[1]$ in some sense, so we might
try to construct it by descent. The problem is that since we are somehow working
with $\Z/p^n$ coefficients for all $n$ at the same time, there will usually not
be any étale cover after which the equivalence $\Sptw[1] \wequi \1[1]$ is
achieved. Indeed we expect this to happen after \emph{all} $p^n$-th roots of unity have
been adjoined for all $n$, and this does not constitute an étale cover. It is
however a \emph{pro-étale} cover. This suggests that we might wish to employ the
\emph{pro-étale topology}, as defined by Bhatt-Scholze \cite{bhatt2013pro}. We
review this somewhat technical notion at the beginning of Section
\ref{sec:proet}, but the upshot is that $\Zptw := \lim_n \mu_{p^n} \in \Sch_X$
belongs to the pro-étale site $X_\proet$, and we define \[ \Sptw[1] :=
\Sigma^\infty K(\Zptw, 1) \in \SH(X_\proet^\hyp)_p^\comp \]
(note that $\Zptw$ is a pro-étale form of $\Z_p$, hence $\Sptw[1]$ is a pro-étale
form of $\Sigma^\infty K(\Z_p,1) \wequi \1[1] \in \SH(X_\proet^\hyp)_p^\comp$).
This object has the expected
properties, but it lives in the wrong category. However one may show that in
good cases (e.g. $X = Spec(\Z[1/p])$), the functor $\SH(X_\et^\hyp)_p^\comp \to
\SH(X_\proet^\hyp)_p^\comp$ is fully faithful and $\Sptw[1]$ is in the essential
image. This way we obtain $\Sptw[1] \in \SH(Spec(\Z[1/p])_\et^\hyp)_p^\comp$,
and we define it over a general scheme by base change from $Spec(\Z[1/p])$.

With this preliminary out of the way, our proof is actually a fairly
straightforward adaptation of the arguments from \cite{cisinski2013etale}. We
can summarize it as follows.
\begin{enumerate}
\item Construct $\Sptw \in \SH(X_\et^\hyp)_p^\comp$.
\item Construct a natural map $\sigma: \Gm \to e(\Sptw[1]) \in
  \SH(\Sm_{X,\et}^\hyp)_p^\comp$ and prove that if $E \in
  \SH(X_\et^\hyp)_p^\comp$ then $e(E)$ is local with respect to the family of
  maps $\sigma \wedge \id_Y$, $Y \in \Sm_X$.
\item Prove homotopy invariance and proper base change for $E \in
  \SH(X_\et^\hyp)$.
\item Prove that $\Sigma^\infty \sigma \in \SH_\et(X)_p^\comp$ is an equivalence.
\end{enumerate}
Once these steps are achieved, we conclude from (4) that $\SH_\et(X) \wequi
L_{\A^1,\sigma} \SH(\Sm_{X,\et}^\hyp)_p^\comp$, where the right hand side
denotes the localization at the family of maps from (2) and also at $Y \times
\A^1 \to Y$. Steps (2,3) then imply that $e: \SH(X_\et^\hyp)_p^\comp \to
\SH_\et(X)_p^\comp$ is fully faithful. Essential surjectivity follows from the fact that
both sides satisfy proper base change (as established for the left hand side in
(3) and for the right hand side by Ayoub), via an argument of Cisinski-Déglise
\cite[Proof of Theorem 4.5.2]{cisinski2013etale}.

Of steps (2--4), the most interesting one is probably (4). Since $\sigma$ is
stable under base change, using a localization argument we may reduce to the
case where $X$ is the spectrum of a separably closed field of characteristic
$\ne p$. In this situation we construct a map $\tau: \Sptw[1] \wequi \1[1] \to
\Gm \in L_{\A^1} \SH(\Sm_{k,\et}^\hyp)_p^\comp$ which induces an inverse to
$\sigma$ in $\DM_\et(k, \Z/p)$. Since $[\1, \1]_{L_{\A^1}
\SH(\Sm_{k,\et}^\hyp)_p^\comp} \wequi \Z_p$ by (3), we deduce from this that
$\sigma\tau: \1[1] \to \1[1] \in L_{\A^1} \SH(\Sm_{k,\et}^\hyp)_p^\comp$ is an equivalence. This
implies that $\Sigma^\infty \sigma$ is an equivalence by a general result about
symmetric monoidal categories \cite[Lemma 22]{bachmann-hurewicz}.

\subsection*{Organization}
In Section \ref{sec:prelim} we collect some preliminaries about $p$-completion,
spectral sheaves, and étale cohomological dimension. In Section
\ref{sec:proet} we use the pro-étale topology to construct the twisting spectrum
$\Sptw$ and establish its properties, achieving step (1). In Section
\ref{sec:etale-cohomology} we prove some ``standard facts'' about étale
cohomology with spectral coefficients. This achieves steps (2) and (3). In
Section \ref{sec:motivic-category} we prove/recall some essentially well-known facts
about the functor $X \mapsto \SH_\et(X)$. Then we carry out step (4) and
hence conclude the proof of the main theorem in Section \ref{sec:main}. We
collect some applications in Section \ref{sec:applications}.

\subsection*{Necessity of the étale finiteness hypothesis}
We prove our main result (and hence all applications) under the assumption of
(local) ``$p$-étale finiteness''. While this is satisfied quite often in practice,
it is an unsatisfying hypothesis, since the rigidity theorems of
Cisinski-Déglise and Ayoub do not need it. Essentially the only part of the
proof where we need the hypothesis is in step (3). That is to say, the author
has been unable to prove (for example) that $\SH(X_\et^\hyp)_p^\comp \to
\SH(\A^1 \times X_\et^\hyp)_p^\comp$ is fully faithful (for $1/p \in X$) without
the assumption that $X$ is étale finite.

We also often use the notion of ``uniformly bounded étale cohomological
dimension $\le n$'', which is slightly stronger than the usual notion of
``locally of étale cohomological dimension $\le n$''. Some of our intermediate
results can probably be strengthened to hold under the weaker assumption; we
chose not to do this because all our examples satisfy the stronger
conclusion anyway.

\subsection*{Use of $\infty$-categories}
This article is written in the language of $\infty$-categories, as set out in
\cite{lurie-htt,lurie-ha}. This is mostly inconsequential: apart from Sections
\ref{sec:proet} and \ref{sec:motivic-category}, everything can be formulated at
the level of triangulated categories, and for Sections \ref{sec:proet} and
\ref{sec:motivic-category} a translation into the language model categories is straightforward.

\subsection*{Notation}
We denote by $\Map(A, B)$ the mapping space between objects in an
$\infty$-category, by $\map(A, B)$ the mapping spectrum in a stable
$\infty$-category, and by $\imap(A, B)$ the internal mapping spectrum in a closed
symmetric monoidal stable $\infty$-category. We put $[A, B] = \pi_0 \Map(A, B)$.
We write $D(X)$ for the strong dual of an object $X$ in a symmetric monoidal
category, if it exists.

We use homological notation for $t$-structures, see e.g. \cite[Section
1.2.1]{lurie-ha}. For any $\infty$-category $\scr C$ (other than stable
$\infty$-categories with a $t$-structure), we denote by $\scr C_{\le
0}$ the subcategory of $0$-truncated objects \cite[Definition
5.5.6.1]{lurie-htt}. In particular for a site $\scr C$,
$\Shv(\scr C)_{\le 0}$ denotes the $1$-category of sheaves of sets on $\scr C$.

\subsection*{Acknowledgements}
I would like to thank Fabien Morel for posing this problem, Marc Hoyois for
suggesting that the ``mystery twisting spectrum'' might have something to do
with ``$\lim_n K(\mu_{p^n}, 1)$'' and Thomas Nikolaus for suggesting that I
should study the pro-étale topology. I would also like to thank Maria Yakerson
for comments on a draft.

\section{Preliminaries}
\label{sec:prelim}

We collect some essentially well-known results.

\subsection{$p$-completion}
\label{subsec:p-completion}

Throughout we fix a stable, presentably symmetric monoidal
$\infty$-category $\scr C$ and a (strongly) dualizable
object $A$. For closely related results, see \cite[Sections 2 and
3]{mathew2017nilpotence}.
We let $\scr C[A^{-1}] = \{X \mid X \otimes A \wequi 0 \} \subset
\scr C$. Denote by $\scr C_{A-tors}$ the left orthogonal of $\scr C[A^{-1}]$,
and by $\scr C_A^\comp$ the right orthogonal.

\begin{lemma}
The inclusion $\scr C[A^{-1}] \subset \scr C$ is both reflective and
co-reflective.
\end{lemma}
\begin{proof}
The functor $\otimes A: \scr C \to \scr C$ has both a right
and a left adjoint (namely tensoring with $DA$), hence preserves limits and
colimits. It follows that $\scr C[A^{-1}]$ is presentable (e.g. use
\cite[Proposition 5.5.3.12]{lurie-htt}) and the inclusion
preserves limits and colimits. The result follows now by the adjoint functor
theorem \cite[Corollary 5.5.2.9]{lurie-htt}.
\end{proof}

By \cite[Remark 6]{barwick2016stablerecollement}, we thus have a recollement
situation, and in particular there is a canonical equivalence $\scr C_{A-tors}
\wequi \scr C_A^\comp$. We can identify $\scr C_{A-tors}, \scr C_A^\comp$ more
explicitly.

\begin{lemma}
\begin{enumerate}
\item The category $\scr C_{A-tors} \subset \scr C$ is the localising
  subcategory generated by objects of the form $DA \otimes X$ for $X \in \scr C$.
\item The category $\scr C_A^\comp$ is the localization of $\scr C$ at the maps
  $f: X \to Y \in \scr C$ such that $f \otimes A$ is an equivalence.
\end{enumerate}
\end{lemma}
\begin{proof}
(1) Let $\scr C'$ be the localizing subcategory generated by objects of the form
$DA \otimes X$. Clearly $\scr C' \subset \scr C_{A-tors}$. The inclusion $\scr
C' \to \scr C$ has a right adjoint by presentability. In order to conclude that
$\scr C' = \scr C_{A-tors}$ it suffices to prove that if $X \in \scr C_{A-tors}$
and $[Y \otimes DA, X] = 0$ for all $Y \in \scr C$, then $X \wequi 0$. The assumption
implies that $X \otimes A \wequi 0$, and hence $X \in \scr C[A^{-1}]$. The
result follows.

(2) The functor $\Fun(\Delta^1, \scr C) \to \Fun(\Delta^1, \scr C), f \mapsto f
\otimes \id_{A}$ is accessible. Hence the localization exists; denote it by
$L\scr C \subset \scr C$. Note that $f \otimes A$ is an equivalence if and only
if $cone(f) \otimes A \wequi 0$. Consequently $L\scr C$ is the right orthogonal
of $\scr C[A^{-1}]$. This concludes the proof.
\end{proof}

\begin{example}
Any stable symmetric monoidal $\infty$-category receives a symmetric monoidal
functor from finite spectra. In particular the case $A = \1/p$ always applies.
In this case $\scr C[A^{-1}]$ consists of the uniquely $p$-divisible objects,
$\scr C_{p-tors} := \scr C_{A-tors}$ consists of the $p$-torsion objects,
and $\scr C_p^\comp := \scr C_A^\comp$ consists of the $p$-complete objects.
In particular we see that if $\scr C$ is compactly generated then so is $\scr
C_{p-tors}$, and hence so is the equivalent category $\scr C_p^\comp$. We call
the maps $f$ such that $f \otimes A$ is an equivalence, i.e. such that $f/p$ is
an equivalence, \emph{$p$-equivalences}.
\end{example}

We have the following obvious but comforting results.

\begin{lemma}
Let $F: \scr C \to \scr D$ be any stable functor of stable $\infty$-categories.
Then $F$ preserves $p$-equivalences.
\end{lemma}
\begin{proof}
A map $\alpha: X \to Y$ is a $p$-equivalence if and only if $\alpha/p: X/p \to
Y/p$ is an equivalence. Here $p: X \to X$ denotes the sum of $p$ times the
identity map, and $X/p$ the cofiber. Similarly for $Y$. Being stable, $F$
preserves the identities, the $Ab$-enrichment of the homotopy categories, and
cofibers. The result follows.
\end{proof}

Consequently, $F$ canonically induces a functor $F_p^\comp = F: \scr C_p^\comp \to \scr
D_p^\comp$.

\begin{lemma} \label{lemm:p-completion-classical-subcat}
Let $F: \scr C \adj \scr D: G$ be an adjunction of stable $\infty$-categories,
with $F$ fully faithful. Then there is an induced adjunction $F_p^\comp: \scr
C_p^\comp \adj \scr D_p^\comp: G_p^\comp$, with $F_p^\comp$ still fully
faithful. The essential image of $F_p^\comp$ consists of those $X \in \scr
D_p^\comp$ such that $X/p \in \scr D$ lies in the essential image of
$F$.
\end{lemma}
\begin{proof}
Only the last statement requires proof. If $X = F_p^\comp(Y)$ then $X/p =
F(Y)/p$ is of the claimed form. Conversely, let $X \in \scr D_p^\comp$ with
$X/p$ in the essential image of $F$. We wish to show that $F_p^\comp G_p^\comp X \to X$ is an
equivalence, for which it suffices to show that it is a $p$-equivalence.
Since $F$, $G$ commute with taking the cofiber of multiplication by $p$, it thus
suffices to show that $FGX/p \to X/p$ is an equivalence. This is
true by assumption.
\end{proof}

\subsection{Sheaves of spectra}
\label{subsec:sheaves-of-spectra}
Given an $\infty$-topos $\scr X$ we denote by $\SH(\scr X)$ the category of
spectral sheaves on $\scr X$, i.e. limit-preserving functors $\scr X^\op \to
\SH$. This is a presentable $\infty$-category \cite[Remark 1.3.6.1]{lurie-sag}.
The category $\SH(\scr X)$ is equivalent to the stabilization of $\scr X$
\cite[Remark 1.3.3.2]{lurie-sag}, and so is in particular stable. One puts
\[ \SH(\scr X)_{\le 0} = \{E \in \SH(\scr X) \mid \Omega^\infty E \wequi * \}. \]
This defines the non-positive part of a right complete $t$-structure on
$\SH(\scr X)$, with heart the category of abelian group objects in $\scr X_{\le
0}$ \cite[Proposition 1.3.2.1]{lurie-sag}. We write $\ul{\pi}_i(E)$ for the
homotopy sheaves of $E \in \SH(\scr X)$. The $t$-structure is nondegenerate
provided that $\scr X$ is hypercomplete (see \cite[Section 6.5]{lurie-htt} for
this notion); if $\scr X$ is furthermore
locally of cohomological dimension $\le n$ for some $n$ then the $t$-structure
is left complete \cite[Corollary 1.3.3.11]{lurie-sag}. If $\scr X$ is an
$\infty$-topos we denote by $\scr X^\hyp$ its hypercompletion; somewhat
abusively if $\scr C$ is a site then we write $\Shv(\scr C^\hyp)$ instead of
$\Shv(\scr C)^\hyp$. Similarly $\SH(\scr C^\hyp)$ means $\SH(\Shv(\scr C^\hyp))
= \SH(\Shv(\scr C)^\hyp)$.
If $g^*: \scr X \to \scr Y$ is (the
left adjoint of) a geometric morphism, then there is an induced adjunction $g^*:
\SH(\scr X) \adj \SH(\scr Y): g_*$, with $g^*$ $t$-exact \cite[Remark
1.3.2.8]{lurie-sag}.

Recall the notion of \emph{coherent} and \emph{locally coherent} $\infty$-topoi
from \cite[Definition A.2.1.6]{lurie-sag}. Coherence is stable under
hypercompletion \cite[Proposition A.2.2.2]{lurie-sag}. If $\scr C$ is a
``finitary'' site (see \cite[Section A.3.1]{lurie-sag}), then $\Shv(\scr C)$ is
locally coherent and coherent \cite[Proposition A.3.1.3]{lurie-sag}. In
particular any ``reasonable'' topology on schemes yields a locally coherent topos,
and any object represented by a qcqs scheme is coherent.

It is well-known that if $\scr X$ is a locally coherent $\infty$-topos and $X \in \scr X$
is coherent, then sheaf cohomology on $X$ commutes with filtered colimits. The
next result is an equally well-known generalization of this.

\begin{lemma} \label{lemm:SH--compact}
Let $\scr X$ be a locally coherent topos and $X \in \scr X$ coherent.
Then $\map(\Sigma^\infty X_+, \bullet): \SH(\scr X)_{\le 0} \to
\SH$ commutes with filtered colimits.
\end{lemma}
\begin{proof}
Immediate consequence of the same statement for $n$-truncated spaces
\cite[Corollary A.2.3.2(1)]{lurie-sag}, using that $\Omega^\infty: \SH(\scr X)
\to \scr X$ preserves filtered colimits (since filtered colimits commute with
finite limits in any $\infty$-topos \cite[Example 7.3.4.7]{lurie-htt}, filtered colimits of
$\Omega$-spectra are $\Omega$-spectra).
\end{proof}

Recall that an object $X$ in an $\infty$-topos $\scr X$ is said to have
cohomological dimension $\le n$ for a family of sheaves $\{\scr F_\alpha \in
\SH(\scr X)^\heart\}_\alpha$ if $H^i(X, \scr F_\alpha) = 0$ for all $i > n$ and
all $\alpha$ \cite[Definition 2.8]{clausen2019hyperdescent}.
Recall also that $E \in \SH(\scr X)$ is said to be \emph{postnikov-complete} if
the natural map $E \to \lim_i E_{\le i}$ is an equivalence.

\begin{lemma} \label{lemm:SH-compact-objects}
Let $\scr X$ be an $\infty$-topos, $Y \in \SH(\scr X)$ postnikov complete, $X \in \scr X$ of
cohomological dimension $\le n$ for $\{\ul{\pi}_i(Y)\}_{i \in \Z}$. Then
\begin{enumerate}
\item We have $[\Sigma^\infty X_+, Y] \wequi [\Sigma^\infty X_+, Y_{\le n}]$.
\item If $Y \in \SH(\scr X)_{\ge m}$ then $\map(\Sigma^\infty X_+, Y) \in \SH_{\ge m-n}$.
\end{enumerate}
Suppose that postnikov towers converge in $\scr X$ and $X$ is of cohomological
dimension $\le n$.
Then also
\begin{enumerate}
\setcounter{enumi}{2}
\item $\Sigma^\infty X_+ \in \SH(\scr X)$ is compact.
\end{enumerate}
Alternatively, suppose that objects in $\SH(\scr X)$ of the form $Y/p$ are
postnikov complete, and $X$ is of $p$-cohomological dimension $\le n$.
Then
\begin{enumerate}
\setcounter{enumi}{3}
\item $\Sigma^\infty X_+/p \in \SH(\scr X)$ is compact.
\end{enumerate}
\end{lemma}
\begin{proof}
We will conflate $X$ and $\Sigma^\infty X_+$ for notational simplicity.

(1) By assumption $Y \wequi \lim_i Y_{\le i}$. Consequently we have the Milnor exact
sequence $0 \to \lim^1_i \pi_1 \Map(X, Y_{\le i}) \to \pi_0 \Map(X, Y) \to \lim
\pi_0 \Map(X, Y_{\le i}) \to 0$ \cite[Proposition
VI.2.15]{goerss2009simplicial}. It is thus enough to show that $\{\pi_1 \Map(X, Y_{\le
i})\}_i$ stabilizes, and $\{\pi_0 \Map(X, Y_{\le i})\}_i$ stabilizes at $i=n$. The
first statement follows from the second applied to $\Omega Y$. To prove the
second statement, it suffices to show that $[X, Y_{\le i+1}] \to [X, Y_{\le i}]$ is an
isomorphism for $i \ge n$. We have the fiber sequence $\ul{\pi}_{i+1}(Y)[i+1]
\wequi (Y_{\le i+1})_{\ge i+1} \to Y_{\le i+1} \to Y_{\le i}$ which induces an
exact sequence \[ [X, \ul{\pi}_{i+1}(Y)[i+1]] \to [X, Y_{\le i+1}] \to [X,
Y_{\le i}] \to [X, \ul{\pi}_{i+1}(Y)[i+2]]. \] The two outer terms vanish for $i
\ge n$ by assumption, whence the result.

(2) We need to prove that $[X[i], Y] = 0$ for $i < m-n$, or equivalently $[X,
Y[-i]] = 0$, or equivalently $[X, Y'] = 0$ for $Y' \in \SH(\scr
X)_{>n}$. But $[X, Y'] = [X, Y'_{\le n}] = [X, 0] = 0$, by (1). The result
follows.

(3) Let $\{Y_i\}_i$ be a filtered system in $\SH(\scr X)$. We have \[ [X, \colim_i
Y_i] \wequi [X, \tau_{\le n} \colim_i Y_i] \wequi [X, \colim_i \tau_{\le n} Y_i]
\wequi \colim_i [X, \tau_{\le n} Y_i] \wequi \colim_i [X, Y_i], \]
which is the desired result. Here we have used (1) for the first and last
equivalence, for the second equivalence we use that $\tau_{\le n}: \SH(\scr X)
\to \SH(\scr X)$ preserves filtered colimits by \cite[Proposition
1.3.2.7(2)]{lurie-sag}, and the third equivalence is Lemma
\ref{lemm:SH--compact}.

(4) Essentially the same argument applies, using that $[X/p, Y] \wequi [X,
Y/p[-1]]$ and $Y/p$ has $p^2$-torsion homotopy sheaves.
\end{proof}

If the assumption of condition (4) holds (i.e. every spectrum of the form $E/p$
is postnikov complete), then we shall say that $\SH(\scr X)$ is $p$-postnikov
complete.
Let us note the following fact.
\begin{remark} \label{rmk:postnikov-convergence}
Suppose that $\scr X$ is hypercomplete and there exists $d$ such that $\scr X$
is locally of ($p$-)cohomological dimension $\le d$.
Then $\SH(X)$ is ($p$-)postnikov complete.
This follows from \cite[Proposition 2.10]{clausen2019hyperdescent}.
\end{remark}

If $f: \scr C \to \scr D$ is a functor of $t$-categories, we call $f$ \emph{of
finite cohomological dimension} ($\le n$) if $f(\scr C_{\ge 0}) \subset \scr
D_{\ge -n}$. We will similarly say that $f$ is \emph{of finite $p$-cohomological
dimension} ($\le n$) if $f(\scr C_{\ge 0}/p) \subset \scr D_{\ge -n}$.
For example in the situation of Lemma
\ref{lemm:SH-compact-objects}, the functor $\map(\Sigma^\infty X_+, \bullet):
\SH(\scr X) \to \SH$ is of ($p$-)cohomological dimension $\le n$.

\begin{lemma} \label{lemm:coh-dimension-trick}
Let $f: \scr C \to \scr D$ be a stable functor between $t$-categories. Assume that $f$
is of finite ($p$-)cohomological dimension and vanishes on bounded above objects
(modulo $p$), and that $\scr D$ is non-degenerate (on $p$-torsion objects). Then $f \wequi 0$.
\end{lemma}
\begin{proof}
Let $E \in \scr C$ and put $E' = E$ (respectively $E'=E/p$). For any
$i$ we have the fiber sequence $E'_{\ge i} \to E' \to E'_{<i}$ and hence we get
a fiber sequence $f(E'_{\ge i}) \to f(E') \to 0$. We conclude that $f(E')$ is $\infty$-connective
($f$ being of finite cohomological dimension), and hence zero.
\end{proof}

\subsection{Étale cohomological dimension}

We start with the following variant of \cite[Theorem 1.1.5]{cisinski2013etale}.
Here and everywhere in this article, cohomological dimension of schemes refers
to étale cohomological dimension.
\begin{theorem}[Gabber, Cisinski-D{\'e}glise]
\label{thm:noetherian-etale-dimension-bound}
Let $X$ be a quasi-separated, noetherian scheme of dimension $d$. Let $d' =
\pcd_p(X) := \sup_{x \in X} \cd_p(k(x))$ and $\pcd(X) = \sup_p \pcd_p(X)$.
Then $\cd_p(X) \le 3d + d' + 3$ and $\cd(X) \le 3d + \pcd(X) + 3$.
\end{theorem}
\begin{proof}
We use the ideas of \cite[Theorem 1.1.5]{cisinski2013etale}. Let $D = 3d + d' +
3$ (respectively $D= 3d + \pcd(X) + 3$).
We need to show that $\H^i_\et(X, F) = 0$ for every $p$-torsion (respectively
every) étale sheaf (of abelian groups) $F$ on $X$ and every $i > D$. Since $X$ is qcqs, étale cohomology commutes
with filtered colimits of sheaves. Hence we may assume that $F$ is
constructible. As in the reference, we may reduce to $F$ being a sheaf of
$\Z/p$-modules (respectively, $\Z/p$-modules for some prime $p$ or $\Q$-modules).
For the case of a $\Q$-module
the claim follows from \cite[Lemma 1.1.4]{cisinski2013etale}. Hence suppose that
$F$ is a $\Z/p$-module. Let $Z = X
\otimes_\Z \Z/p$ and $U = X \setminus Z$. Let $i$ be the closed and $j$ the open
immersion. We have the distinguished triangle (*) $i_* Ri^! F \to F \to Rj_* j^* F$.
From this we get in particular that $R^b j_* j^* F \wequi i_* R^{b+1} i^! F$ for $b
\ge 1$. We have $R^b j_* j^* F = 0$ for $b > 2d+d'$ (see the reference; this
uses \cite[Lemma XVIII-A.2.2]{illusie2014travaux}) and hence $Ri^! F$ is concentrated in
cohomological degrees $\le 2d + d' + 2$. Considering (*), it is enough to show
that $\H_\et^i(Z, Ri^! F) = 0$ and $\H^i_\et(U, j^* F) = 0$ for $i > D$. For $U$, this is
\cite[Lemma XVIII-A.2.2]{illusie2014travaux} again. For $Z$ we use that
$\cd_p(Z) \le \dim{Z} + 1 \le d + 1$ \cite[Theorem X.5.1]{sga4}
(and \cite[Tag 02UZ]{stacks-project}).
\end{proof}

Abstracting from this, we will make use of the following notion.
\begin{definition} \label{def:uniformly-bounded}
We say that a scheme $S$ is of uniformly bounded $p$-étale cohomological dimension $\le n$ if
for every finitely-presented, qcqs étale scheme $Y/X$ we have $\cd_p(Y) \le n$;
uniformly bounded étale cohomological dimension $\le n$ is defined similarly.
\end{definition}

We recall the following well-known facts.
\begin{lemma} \label{lemm:pcd-extension}
Let $f: X \to S$ be a morphism of schemes.
\begin{enumerate}
\item If $f$ is quasi-finite, then $\pcd_p(X) \le \pcd_p(S)$ and $\pcd(X) \le \pcd(S)$.
\item If $f$ is finite type and $S$ is quasi-compact, then $\pcd_p(X) < \infty$
  (respectively $\pcd(X) < \infty$) as soon as $\pcd_p(S) < \infty$ (respectively
  $\pcd(S) < \infty$).
\end{enumerate}
\end{lemma}
\begin{proof}
This follows from the fact that if $L/K$ is an extension of fields, then $\cd_p(L)
\le \cd_p(K) + trdeg(L/K)$ \cite[Theorem 28 of Chapter
4]{shatz1972profinite}.\NB{See \cite[Tag 01TG]{stacks-project} for  finiteness
of residue field extension in (1).}
\end{proof}

\begin{corollary} \label{cor:uniformly-finite-etale-dimension-example}
If $X$ is noetherian, $\dim{X} < \infty$ and $\pcd(X) < \infty$ (respectively
$\pcd_p(X) < \infty$), then $X$ is
of uniformly bounded ($p$-)étale cohomological dimension. If $Y/X$ is finite type,
then also $Y$ is uniformly of bounded ($p$-)étale cohomological dimension.
\end{corollary}
\begin{proof}
By Lemma \ref{lemm:pcd-extension}(2), $Y$ satisfies the same
assumptions as $X$\NB{Locally on $X$, $Y$ can be covered by closed subschemes of
$\A^n_X$. Quasicompactness of $X$ then implies that $Y$ has finite dimension.
$Y$ is noetherian since a f.g. $A$-algebra with $A$ noetherian is noetherian}.
Hence we may assume that $Y=X$.

Let $U/X$ be étale. We have $\pcd_p(U) \le \pcd_p(X)$ by Lemma
\ref{lemm:pcd-extension}(1).
Since also $\dim{U} \le \dim{X}$,
it follows from Theorem \ref{thm:noetherian-etale-dimension-bound}
that if $U$ is moreover quasi-separated and finite type (whence noetherian), we
have $\cd_p(U) \le \pcd_p(X) + 3 \dim{X} + 3$. The argument for $\cd$ in place
of $\cd_p$ is the same. This concludes the proof.
\end{proof}

\begin{example} \label{ex:uniformly-bounded-dim}
The assumptions of Corollary \ref{cor:uniformly-finite-etale-dimension-example}
hold for $X = Spec(R)$, where $R$ is a field of finite étale cohomological
dimension (e.g. separably closed fields, finite fields, unorderable number
fields), or a strictly henselian noetherian
local ring \cite[Lemma XVIII-A.1.1]{illusie2014travaux}.
They hold étale-locally on $X = Spec(R)$ for $R$ a field of finite virtual étale
cohomological dimension (e.g. number fields), or $R = \Z$ (cover $Spec(\Z)$ by
$Spec(\Z[1/2, x]/x^2 + 1) \coprod Spec(\Z[1/3, x]/x^2 + x + 1)$\NB{The
discriminants are -4 and -3.}).
\end{example}

Let us note the following permanence property.

\begin{lemma} \label{lemm:uniform-dim-permanence}
Let $\{S_i\}_i$ be a pro-(qcqs scheme) with finitely presented, affine étale transition morphisms.
If $S_0$ is of uniformly bounded ($p$-)étale cohomological dimension, then so is
$S := \lim_i S_i$.
\end{lemma}
\begin{proof}
Suppose $S_0$ is of uniformly bounded $p$-étale cohomological dimension $\le n$.
Then the same holds for $S_i$ for all $i$. If $X/S$ is étale and qcqs,
then $X = \lim_i X_i$ for some system of qcqs étale schemes $X_i/S_i$.
Since étale cohomology of qcqs schemes commutes with cofiltered limits
\cite[Theorem VII.5.7]{sga4}, we conclude that $\cd_p(X) \le n$. The argument
for étale cohomological dimension is the same. The result
follows.
\end{proof}

Let us summarize the following convenient properties.
\begin{lemma} \label{lemm:uniformly-bounded-compact-generation}
Let $X$ be (étale locally) of uniformly bounded étale cohomological dimension.
Then $\SH(X_\et^\hyp)$ is postnikov-complete and
compactly generated. In fact if $Y \in X_\et$ is qcqs (of finite étale
cohomological dimension), then $\Sigma^\infty_+ Y$ is compact.

If $X$ is only (étale locally) of uniformly bounded $p$-étale cohomological
dimension, then $\SH(X_\et^\hyp)$ is $p$-postnikov complete and
$\SH(X_\et^\hyp)_p^\comp$ is compactly generated. In fact if $Y \in X_\et$ is
qcqs (of finite $p$-étale
cohomological dimension), then $\Sigma^\infty_+ Y/p \in \SH(X_\et^\hyp)_p^\comp$
is compact.
\end{lemma}
\begin{proof}
The ($p$-)postnikov completeness follows from Remark
\ref{rmk:postnikov-convergence}.
Lemma \ref{lemm:SH-compact-objects}(3,4) now show that $\Sigma^\infty_+Y$
(respectively $\Sigma^\infty_+Y/p$) is compact in $\SH(X_\et^\hyp)$.
We deduce that $\Sigma^\infty_+Y/p \in \SH(X_\et^\hyp)_p^\comp$ is compact, as
follows:
\begin{align*}
[Y/p, \colim_i T_i]_{\SH(X_\et^\hyp)_p^\comp}
  &\wequi [Y[1], (\colim_i T_i)/p]_{\SH(X_\et^\hyp)_p^\comp} \\
  &\wequi [Y[1], (\colim_i T_i)/p]_{\SH(X_\et^\hyp)} \\
  &\wequi [Y/p, \colim_i T_i]_{\SH(X_\et^\hyp)} \\
  &\wequi \colim_i [Y/p, T_i]_{\SH(X_\et^\hyp)} \\
  &\wequi \colim_i [Y[1], T_i/p]_{\SH(X_\et^\hyp)} \\
  &\wequi \colim_i [Y[1], T_i/p]_{\SH(X_\et^\hyp)_p^\comp} \\
  &\wequi \colim_i [Y/p, T_i]_{\SH(X_\et^\hyp)_p^\comp}
\end{align*}
In other words we use that the forgetful functor $\SH(X_\et^\hyp)_p^\comp \to
\SH(X_\et^\hyp)$ commutes with the formation of $(\colim_i \ph)/p$.
\end{proof}

\section{Construction of the twisting spectrum}
\label{sec:proet}

We will make use of some of the convenient properties of the pro-étale topology
\cite{bhatt2013pro}. Recall that a morphism of schemes $X \to Y$ is called
\emph{weakly étale} if $X \to Y$ is flat and $\Delta: X \to X \times_Y X$ is
also flat. The pro-étale site of $X$ consists of those weakly étale $X$-schemes
of cardinality smaller than some (large enough) bound; the coverings are the
fpqc coverings. We denote the pro-étale topos of $X$ (with respect to the
implicit cardinality bound) by $X_\proet$.

We spell out some salient properties.

\begin{lemma} \label{lemm:weakly-contractible-dim}
Let $\scr X$ be a topos.
If $X \in \scr X$ is a weakly contractible \cite[Definition 3.2.1]{bhatt2013pro}
object, then $X$ has cohomological dimension $0$.
\end{lemma}
\begin{proof}
It suffices to prove that $\H^0(X, \bullet)$ is exact. Thus let $f: F \to G$ be a
surjection of sheaves and $a \in G(X)$. Then there is a covering $\{U_\alpha \to
X\}_\alpha$ and $\{a_\alpha \in F(U_\alpha)\}_\alpha$ such that $f(a_\alpha) =
a|_{U_\alpha}$. Let $\tilde{X} = \coprod_\alpha U_\alpha$, where the coproduct
is taken in the 1-topos $\scr X$. We hence obtain
$\tilde a \in F(\tilde{X}) \wequi \prod_\alpha F(U_{\alpha})$. Now $p: \tilde{X} \to X$ is a
surjection, hence has a section $s$ (by definition of weak contractibility).
We find that $f(s^* \tilde a) = s^* f(\tilde a) = s^* p^* a = a$. This concludes the proof.
\end{proof}

\begin{remark}[Hoyois]
The notion of a weakly contractible object in a $1$-topos \cite[Definition 3.2.1]{bhatt2013pro}
is essentially the same as that of an object of homotopy dimension $\le 0$
in an $\infty$-topos \cite[Definition 7.2.1.1]{lurie-htt}.
Lemma \ref{lemm:weakly-contractible-dim}
above is essentially a special case of \cite[Corollary 7.2.2.30]{lurie-htt}.
\end{remark}

\begin{corollary}
For any scheme $X$, postnikov towers converge in $\Shv(X_\proet^\hyp)$ and
$\SH(X_\proet^\hyp)$ is left-complete and compactly generated.
\end{corollary}
\begin{proof}
$X_\proet$ is generated by coherent weakly contractible objects
\cite[Proposition 4.2.8]{bhatt2013pro}. Now apply \cite[Corollary
1.3.3.11]{lurie-sag} and Lemmas
\ref{lemm:weakly-contractible-dim} and \ref{lemm:SH-compact-objects}.
\end{proof}

We have a canonical geometric morphism $\nu^*: X_\et \to X_\proet$. The functor
$\nu^*: \Shv(X_\et)_{\le 0} \to \Shv(X_\proet)_{\le 0}$ is fully faithful \cite[Lemma
5.1.2]{bhatt2013pro}. A sheaf $F \in \Shv(X_\proet)_{\le 0}$ is called
\emph{classical} if it is in the essential image of $\nu^*$. Note that a sheaf
is classical if and only if it is continuous for pro-étale systems of affine
schemes \cite[Lemma 5.1.2]{bhatt2013pro}. For a $t$-category $\scr C$, we denote
by $\scr C_{-}$ the subcategory of (homologically) bounded above objects.

\begin{lemma} \label{lemm:proet-ff}
The functor $\nu^*: \SH(X_\et^\hyp)_- \to \SH(X_\proet^\hyp)_-$ is fully
faithful. If postnikov towers converge in $X_\et^\hyp$, then the functor
$\nu^*: \SH(X_\et^\hyp) \to \SH(X_\proet^\hyp)$ is also fully faithful.
\end{lemma}
\begin{proof}
Let $U \in X_\et$ be qcqs (e.g. affine). By Lemma \ref{lemm:SH--compact}, the
functors $\map(\Sigma^\infty_+ U, \bullet): \SH(X_\et^\hyp)_{\le 0}, \SH(X_\proet^\hyp)_{\le 0}
\to \SH$ preserve filtered colimits. Consequently the composite $\alpha = \nu_*
\nu^*: \SH(X_\et^\hyp)_{\le 0} \to \SH(X_\et^\hyp)_{\le 0}$ preserves filtered colimits.
For $E \in \SH(X_\et^\hyp)^\heart$ we have $\alpha(E) \wequi E$ by
\cite[Corollary 5.1.6]{bhatt2013pro}, whence the same holds for bounded spectra.
The result about $\SH(X_\et^\hyp)_-$
follows since this category is generated under filtered colimits by bounded
spectra. In general we have $\nu_* \nu^* E \wequi \nu_*(\lim_i (\nu^*
E)_{\le i})$, which is the same as $\lim_i \nu_* \nu^* E_{\le i}$ since $\nu_*$
preserves limits and $\nu^*$ preserves truncation. This is the same as $\lim_i
E_{\le i}$, which is equivalent to $E$ if postnikov towers converge.
\end{proof}

\begin{proposition} \label{prop:proet-characterise}
The essential image of $\nu^*: \SH(X_\et^\hyp)_- \to \SH(X_\proet^\hyp)_-$
consists of those spectra $E \in \SH(X_\proet^\hyp)_-$ with classical homotopy
sheaves, and the functor $\nu_*$ is $t$-exact when restricted to such objects.

If $X$ has étale-locally uniformly bounded étale cohomological dimension\todo{can we get away
with locally finite dimension?}, then the
functor $\nu^*: \SH(X_\et^\hyp) \to \SH(X_\proet^\hyp)$ has
essential image those spectra with classical homotopy sheaves, and $\nu_*$ is
$t$-exact on the entire essential image of $\nu^*$.
\end{proposition}
\begin{proof}
Clearly spectra in the essential image of $\nu^*$ have classical homotopy sheaves; we need
to prove the converse.

The functor $\nu^*: \SH(X_\et^\hyp)^\heart \to \SH(X_\proet^\hyp)^\heart$ is an
equivalence onto the subcategory of classical sheaves, with inverse given by
$\nu_*$ (by Lemma \ref{lemm:proet-ff}). The result for spectra
with bounded homotopy sheaves follows immediately.
As in the proof of Lemma \ref{lemm:proet-ff}, the functor $\nu_*$ preserves
filtered colimits in $\SH(X_\proet^\hyp)_{\le 0}$, and hence so does $\nu^*
\nu_*$. Since the subcategory of $\SH(X_\proet^\hyp)_{\le 0}$ consisting of
spectra with classical homotopy sheaves is generated under filtered
colimits by bounded spectra with classical homotopy sheaves,
we find that $\nu^*\nu_* \Rightarrow \id$ is an
equivalence on this subcategory. This first statement follows.

Now let $X$ be étale-locally of uniformly bounded étale cohomological dimension.
Postnikov towers converge in $X_\et^\hyp$ (see Lemma
\ref{lemm:uniformly-bounded-compact-generation}). Hence $\nu^*: \SH(X_\et^\hyp) \to
\SH(X_\proet^\hyp)$ is fully faithful by Lemma \ref{lemm:proet-ff}. Let $U
\in X_\et$ have cohomological dimension $< n$ and let $E \in
\SH(X_\proet^\hyp)_{> n}$ have classical homotopy sheaves. I claim that $[U,
\nu_* E] \wequi 0$. To see this, we note that $\Map(U, \nu_* E) \wequi \lim_i \Map(U,
\nu_* (E_{\le i}))$ (since postnikov towers converge in $\SH(X_\proet^\hyp)$)
and hence by the Milnor exact sequence it
suffices to show that $[U[\epsilon], \nu_* (E_{\le i})] = 0$ for all $i$ and $\epsilon
\in \{0,1\}$. But $\nu_* (E_{\le i}) \in \SH(X_\et^\hyp)_{> n}$ by the
$t$-exactness result above, whence the claim follows from Lemma
\ref{lemm:SH-compact-objects}(2). We deduce that if $E \in
\SH(X_\proet^\hyp)$ has classical homotopy sheaves,
then $[U, \nu_* E] \wequi [U, \nu_*(E_{\le n})]$.

Let
$\bar x$ be a geometric point of $X$. Then $X_{\bar x} = \lim_\alpha X_\alpha$,
where $\{X_\alpha\}$ is a cofiltered system of affine étale $X$-schemes of
bounded cohomological dimension (here we use that $X$ is étale-locally
uniformly of bounded
étale cohomological dimension), say bounded by $n \ge 0$. We deduce that if $E
\in \SH(X_\proet^\hyp)$ has classical homotopy sheaves, then
\begin{gather*}
  \ul{\pi}_0(\nu_* E)(X_{\bar x}) = \colim_\alpha [X_\alpha, \nu_* E]
           \wequi \colim_\alpha [X_\alpha, \nu_* (E_{\le n})]
           = \ul{\pi}_0(\nu_*(E_{\le n}))(X_{\bar x}) \\
           \wequi \nu_*(\ul{\pi}_0(E_{\le n}))(X_{\bar x})
           \wequi \nu_*(\ul{\pi}_0(E))(X_{\bar x}),
\end{gather*}
using the $t$-exactness statement for bounded above spectra (which we already
proved) again. We have thus shown that $\nu_*$ is
$t$-exact on arbitrary spectra with classical homotopy sheaves. Since $\nu^*$ is
$t$-exact (as always) and $X_\proet^\hyp$ is hypercomplete (by definition),
this shows that $\nu^* \nu_* \Rightarrow \id$ is an equivalence on spectra with
classical homotopy sheaves, which concludes the proof.
\end{proof}

After these preparatory remarks, we come to our application of the pro-étale
topology: the construction of the twisting spectrum.
Given a sheaf of groups $F$ on $X_\proet$, we denote by $K(F, 1) \in
\Shv(X_\proet^\hyp)$ the sheaf obtained from the presheaf $U \mapsto K(F(U),
1)$, where $K(F(U), 1)$ is the Eilenberg-MacLane space of the group $F(U)$.

\begin{theorem} \label{thm:Sptw-small}
Let $X$ be a scheme and $1/p \in X$. Let $\Zptw := \lim_n \mu_{p^n} \in
\Shv(X_\proet^\hyp)_{\le 0}$ and $\Sptw[1] = (\Sigma^\infty K(\Zptw, 1))_p^\comp
\in \SH(X_\proet^\hyp)_p^\comp$.
\begin{enumerate}
\item $\Zptw$ is a representable by a weakly étale $X$-scheme, and so stable under base change.
\item $\Sptw$ is stable under base change, and an invertible object of
  $\SH(X_\proet^\hyp)_p^\comp$.
\item If $X$ has all $p^n$-th roots of unity for all $n$, then $\Sptw \wequi \1
  \in \SH(X_\proet^\hyp)_p^\comp$.
\item $\Sptw$ lies in the essential image of $\nu^*: \SH(X_\et)_p^\comp \to
  \SH(X_\proet)_p^\comp$. The same holds for its $\otimes$-inverse.
\end{enumerate}
\end{theorem}
\begin{proof}
(1) Attaching $p$-th roots of unity is an étale extension away from $p$
\cite[Corollary 10.4]{neukirch2013algebraic}.
Consequently $\mu_p$ is an étale $X$-scheme, and $\Zptw$ is represented by the
same limit, taken in the category of schemes. Finally pullback of schemes is a limit, so
commutes with limits, so the scheme $\Zptw$ is stable under pullback.

(2) Stability is clear. For invertiblity, note first that if $f: Y \to X$ is weakly
étale, then $f^*$ has a left adjoint $f_\#$ which satisfies a projection
formula\NB{All functors under consideration preserve colimits, so reduce to the
case of schemes, where this is clear.}. It follows that $f^* \imap(A, B) \wequi \imap(f^* A, f^* B)$\NB{$[T,
f^* \imap(A, B)] = [f_\# T, \imap(A,B)] = [f_\# T \otimes A, B] = [f_\#(T
\otimes f^* A), B] = [T \otimes f^*A, f^*B] = [T, \imap(f^*A, f^*B)]$}. This
implies that being invertible is pro-étale local on $X$\NB{$T$ is invertible if
and only if the canonical map $\imap(T, \1) \otimes T \to \1$ is an equivalence}.
Let $X'$ be obtained by attaching all $p^n$-th roots of unity to $X$, for all
$n$. By construction $X' \to X$ is pro-(finite étale), and in fact a covering
map \cite[Tag 090N]{stacks-project}. We may thus replace $X$ by $X'$ and so
assume that $X$ has all $p^n$-th roots of unity for all $n$.
It thus suffices to show (3).

(3) In this situation $\mu_{p^n} \wequi \Z/p^n$
and $\Zptw \wequi \hat\Z_p$ (defined to be $\lim \Z/p^n$ taken in $\Shv(X_\proet^\hyp)$).
Note that $1 \in \hat\Z_p$ defines a map $S^1 \to K(\hat\Z_p, 1)$ which we shall show
is a stable $p$-equivalence.  
As a preparatory remark, let $F \in Ab(X_\proet)$ be any sheaf of abelian
groups. Then $K(F, 1) \in \Shv(X_\proet^\hyp)$ is a sheaf of spaces with
$\pi_0(K(F,1)(Y)) = \H^1_\proet(Y, F)$, $\pi_1(K(F,1)(Y)) = F(Y)$ and
$\pi_i(K(F,1)(Y)) = 0$ else, for all $Y \in X_\proet$. In particular, if $Y$ is
$w$-contractible, then $K(F,1)(Y) \wequi K(F(Y), 1)$
(by Lemma \ref{lemm:weakly-contractible-dim}). Let $f: F \to G \in
Ab(X_\proet)$ and assume that for each $w$-contractible $Y$, the map $F(Y) \to
G(Y)$ is a derived $p$-completion (i.e. $\Hom(\Z/p^\infty, F(Y)) = 0 =
\Hom(\Z/p^\infty, G(Y))$ and $G(Y) \wequi Ext(\Z/p^\infty, F(Y))$).
Then $K(F, 1)(Y) \to K(G, 1)(Y)$ is a
$p$-equivalence in the classical sense \cite[VI.2.2]{bousfield1987homotopy}, and consequently
$\Sigma^\infty (K(F, 1)(Y)) \to \Sigma^\infty (K(G, 1)(Y))$ is a $p$-equivalence
of spectra as follows from \cite[Proposition
5.3(i)]{bousfield1987homotopy}\NB{if $f: X \to Y$ is a p-equivalence of
nilpotent spaces,
i.e. $\Sigma^\infty f \wedge H\Z/p \wequi (\Sigma^\infty f/p) \wedge H\Z$ is an
equivalence, then also $\Sigma^\infty f/p$ is an equivalence}.
This implies that $\Sigma^\infty(f) \in \SH(X_\proet^\hyp)$ is a
$p$-equivalence: it suffices to show that $cof(\Sigma^\infty(f))/p$ has
vanishing homotopy sheaves, which follows from our assumption about
$w$-contractible $Y$ and the fact that $\Sigma^\infty = L_\proet
\Sigma^\infty_{pre}: \Shv(X_\proet^\hyp)_* \to \SH(X_\proet^\hyp)$,
using Lemma \ref{lemm:w-presheaf-zero} below. Here
$\Sigma^\infty_{pre}: \PSh(X_\proet)_* \to \SH(\PSh(X_\proet))$ is the presheaf
level suspension spectrum functor and $L_\proet: \SH(\PSh(X_\proet)) \to
\SH(X_\proet^\hyp)$ is the pro-étale hypersheafification functor.

To apply this remark to our case, note that $S^1 \wequi K(\ul\Z, 1) \in
\Shv(X_\proet^\hyp)_*$, where $\ul \Z \in \Shv(X_\proet^\hyp)_{\le 0}$
denotes the constant sheaf. We have $\H^0_\proet(Y, \Z) = M(|Y|, \Z)$, where $M$
denotes the set of continuous maps between two topological spaces (we view $\Z$
as discrete); see Lemma \ref{lemm:constant-sheaves} below.
Also $\H^0_\proet(Y, \hat\Z_p) = \lim_n \H^0_\proet(Y, \Z/p^n) = \lim_n
M(|Y|, \Z/p^n)$, by the same Lemma (or \cite[Lemma 4.2.12]{bhatt2013pro}).
We thus need to prove that $M(|Y|, \Z) \to \lim_n
M(|Y|, \Z/p^n)$ is a derived $p$-completion of abelian groups. Since the source
has no $p$-torsion, for this it is enough to show that $M(|Y|, \Z) \to M(|Y|,
\Z/p^n)$ is surjective for every $n$ (clearly the kernel is $p^n M(|Y|, \Z)$)
\cite[Section VI.2.1, bottom of page 166]{bousfield1987homotopy}.
But this is clear: if $s: \Z/p^n \to \Z$ is any set-theoretic section of $\Z \to
\Z/p^n$ then composition with $s$ induces a section of $M(|Y|, \Z) \to M(|Y|,
\Z/p^n)$.

(4) It suffices to treat the case $X = Spec(\Z[1/p])$.
By Example \ref{ex:uniformly-bounded-dim}, $X$ has étale-locally uniformly
bounded étale cohomological dimension. Hence by Lemma
\ref{lemm:uniformly-bounded-compact-generation}, Lemma \ref{lemm:proet-ff} and
Proposition \ref{prop:proet-characterise},
$\nu^*: \SH(X_\et^\hyp) \to \SH(X_\proet^\hyp)$ is fully faithful with essential
image the spectra with classical homotopy sheaves. Lemma
\ref{lemm:p-completion-classical-subcat} now implies that it is sufficient (and
necessary) to
show that $\Sptw/p$ and its dual $D(\Sptw)/p$ have classical homotopy sheaves.
By \cite[Lemma 5.1.4]{bhatt2013pro} this is pro-étale local on $X$ (note that
taking strong duals commutes with base change), so we may assume that
$\Sptw$ is $p$-equivalent to $\1$, in which case the claim is clear.
\end{proof}

\begin{lemma} \label{lemm:w-presheaf-zero}
Let $\scr C$ be a site and $F$ a presheaf of pointed sets on $\scr C$.
Suppose that for every $X \in \scr C$ there exists a covering
$\{Y_i \to X\}_i$ with $F(Y_i)=*$ for
all $i$. Then $aF = *$.
\end{lemma}
\begin{proof}
Any pointed map $F \to G$ with $G$ a pointed sheaf must be zero. The result
follows from the Yoneda lemma since $a$ is left adjoint to the inclusion of
sheaves into presheaves.
\end{proof}

\begin{lemma}\todo{surely this must be citable??} \label{lemm:constant-sheaves}
Let $S$ be a set. Define presheaves $F_1, F_2$ on the category of schemes, via $F_1(X) = S$
and $F_2(X) = M(|X|, S)$, the set of continuous maps from the underlying
topological space of $X$ to $S$ (viewed as a discrete topological space). Then
\begin{enumerate}
\item The canonical map $F_1 \to F_2$ is a Zariski equivalence.
\item $F_2$ is a sheaf in the fpqc topology.
\end{enumerate}
\end{lemma}
\begin{proof}
The map $F_1 \to F_2$ is clearly injective. We show that it induces a surjection
on the sheafification, whence (1). To do so, given $f \in F_2(X)$ we have to find a Zariski
cover $\{U_\alpha\}_\alpha$ of $X$ and elements $f_\alpha \in F_1(X)$ with
$f_\alpha$ mapping to $f|_{U_\alpha}$. The image of $F_1 \to F_2$ consists of
the constant functions; hence the cover $X = \coprod_{s \in S} f^{-1}(\{s\})$
works.

Note that $M(|X|, S)$ is the set of locally constant functions from $|X|$ to
$S$. This condition is clearly Zariski local, so $F_2$ is a Zariski sheaf. To
prove that $F_2$ is an fpqc sheaf, it thus suffices to prove that $F_2$ has
descent for faithfully flat morphisms $\alpha: X \to Y$ of affine schemes
\cite[Tag 03O1]{stacks-project}. The canonical map $|X \times_Y X| \to |X| \times_{|Y|}
|X|$ is surjective, as is $X \to Y$;
this implies that an arbitrary function $f: |X| \to S$
descends to $Y$ if and only if the two pullbacks to $|X \times_Y X|$ agree, and
uniquely so; in other
words the presheaf of arbitrary (not necessarily continuous) functions into $S$
is a sheaf. Finally for $U \subset |Y|$, we have that $U$ is open if
and only if $\alpha^{-1}(U)$ is open \cite[Tag 0256]{stacks-project}; this implies that
$f$ is continuous if and only if $f \circ \alpha$ is. This concludes the proof.
\end{proof}

Theorem \ref{thm:Sptw-small} applies in particular if $X = Spec(\Z[1/p])$
in which case $\nu^*$ is fully faithful (by Example
\ref{ex:uniformly-bounded-dim}, Lemma
\ref{lemm:uniformly-bounded-compact-generation} and Lemma \ref{lemm:proet-ff}).
\begin{definition}
We put $\Sptw = \nu_* \Sptw \in \SH(Spec(\Z[1/p])_\et^\hyp)_p^\comp$. For
general schemes $X$ with $1/p \in X$ there is a unique morphism
$f: X \to Spec(\Z[1/p])$ and we define $\Sptw = f^* \Sptw \in \SH(X_\et)_p^\comp$.
\end{definition}

It follows that $\Sptw \in \SH(X_\et)_p^\comp$ is stable under base change,
invertible, and $\nu^* \Sptw = \Sptw$. We offer the following further plausibility
check.

\begin{lemma} \label{lemm:plausibility}
We have $\Sptw \wedge H\Z/p^n \wequi \mu_{p^n} \in \SH(X_\et^\hyp)^\heart$.
\end{lemma}
\begin{proof}
For $E \in \SH(X_\et^\hyp)$, let $\ul{h}_i(E, \Z/p^n) = \ul{\pi}_i(E \wedge H\Z/p^n)$.
By hypercompleteness,
what we have to show is the following: $\ul{h}_i(\Sptw, \Z/p^n) = 0$ for $i \ne
0$, and $\ul{h}_0(\Sptw, \Z/p^n) \wequi \mu_{p^n}$.
The first condition we can check on the
stalks, so assume that $X$ has all $p^m$-th roots of unity for all $m$. Then
$\Sptw \wequi \1_p^\comp$ and so the claim is clear.

To determine $\ul{h}_0(\Sptw, \Z/p^n)$, we may work in
$\SH(X_\proet^\hyp)$ instead (since $\nu^*$ is $t$-exact and $\nu^{*\heart}$ is
fully faithful).
We can model $K(\Zptw, 1)$ by the bar construction on
$\Zptw$. This implies that the homotopy sheaves of $\Sptw[1] \wedge H\Z/p^n$
are given by the
sheafifications of $U \mapsto \tilde{H}_i(\Zptw(U), \Z/p^n)$, where on the right
hand side we mean ordinary (reduced) group homology. Since $\tilde{H}_1(A, \Z/p^n) = A/p^n
A$ for $A$ any abelian group, we find that $\ul{h}_1(K(\Zptw, 1), \Z/p^n) =
\Zptw/p^n \wequi \mu_{p^n}$ (where the last isomorphism can be checked pro-étale locally, whence
assuming that there are all roots of unity\NB{Since $\Z/p^n$ is discrete, the topological ring homomorphism $\Z_p
\to \Z/p^n$ has a continuous section. This implies that $0 \to M(X, \Z_p)
\xrightarrow{p^n} M(X, \Z_p) \to M(X, \Z/p^n) \to 0$ is exact for any
topological space $X$. This shows that $0 \to \Zptw(1) \xrightarrow{p^n} \Zptw(1)
\to \mu_{p^n} \to 0$ is exact when evaluated on any scheme with all roots of
unity (e.g. $w$-contractible ones).}).
This concludes the proof.
\end{proof}

Suppose $S$ is a scheme and $1/p \in S$. Let $C_n: (\A^1 \setminus 0)_S
\to (\A^1 \setminus 0)_S$ be given by raising
the coordinate to the $p^n$-th power. Since $1/p \in S$ this is étale. For each
$n$ we have a commutative diagram
\begin{equation*}
\begin{CD}
(\A^n \setminus 0)_S @>C_1>> (\A^n \setminus 0)_S \\
@VC_{n+1}VV                  @VC_nVV           \\
(\A^n \setminus 0)_S @=         (\A^n \setminus 0)_S.
\end{CD}
\end{equation*}
Hence we obtain an inverse system $\{C_n\}_n$ over $(\A^1 \setminus 0)_S$ with
limit $C = C_S := \lim_n C_n \in (\A^1 \setminus 0)_{S, \proet}$.

\begin{proposition} \label{prop:sigma}
The object $C$ is canonically a $\Zptw$-torsor, and hence is classified by an element
$\sigma = \sigma_S \in [*, K(\Zptw, 1)]_{\Shv((\A^1 \setminus 0)_{S,
\proet}^\hyp)}$.
It is stable under base change. Moreover if $i_1: S \to (\A^1 \setminus 0)_S$ is
the inclusion at the point $1$, then $i_1^*(\sigma) = *$.
\end{proposition}
\begin{proof}
Since $C$ is representable, it is stable under base change. Note that $C_n =
S[t, t^{-1}, u]/(u^{p^n} - t)$; it follows immediately that $i_1^* C_n =
\mu_{p^n}$ and so $i_1^* C$ is the trivial torsor. It remains to explain
the $\Zptw$-torsor structure. We have the multiplication map $(\A^1 \setminus 0)
\times (\A^1 \setminus 0) \to \A^1 \setminus 0$. Restricting the first factor
to $\mu_{p^n} \subset \A^1 \setminus 0$ we obtain $\mu_{p^n} \times C_n \to C_n$.
The structure map $C_n \to \A^1 \setminus 0$ is equivariant for the trivial
action by $\mu_{p^n}$ on the target. Taking the inverse limit we obtain an
action $\Zptw \times C \to C$, and the structure map $C \to \A^1 \setminus 0$ is
equivariant. To prove that this is a torsor, it remains to show that the
shearing map $\Zptw \times C \to C \times_{\A^1 \setminus 0} C$ is an
isomorphism. Since limits commute, for this it is enough to show that each $C_n$
is a $\mu_{p^n}$-torsor, which is clear.
\end{proof}

Upon stabilization and $p$-completion, we obtain a map $\sigma' = (\Sigma^\infty
\sigma)_p^\comp: \1 \to \Sptw[1] \in \SH((\A^1 \setminus 0)_{S, \proet}^\hyp)_p^\comp$,
stable under base change. If $\A^1_S$ is étale-locally of uniformly bounded étale cohomological
dimension, e.g. $S = Spec(\Z[1/p])$, by Lemma \ref{lemm:proet-ff}
there is a unique map $\sigma: \1 \to \Sptw[1] \in \SH((\A^1 \setminus
0)_{S, \et}^\hyp)_p^\comp$ with $\nu^*(\sigma) = \sigma'$. This map $\sigma$
is also stable under base change whenever defined.

\begin{definition} \label{def:sigma}
Let $S/\Z[1/p]$ be a base scheme.
We denote by \[ \sigma: \1 \to \Sptw[1] \in \SH((\A^1 \setminus 0)_{S, \et}^\hyp)_p^\comp \] 
the map obtained from the one constructed above by base change to $S$.
\end{definition}

\section{Complements on étale cohomology}
\label{sec:etale-cohomology}

\begin{lemma} \label{lemm:etale-base-change}
Consider a cartesian square
\begin{equation*}
\begin{CD}
Y' @>g'>> Y   \\
@Vf'VV  @VfVV \\
X' @>g>> X,
\end{CD}
\end{equation*}
with $g$ étale.
Then \[ g^*f_* \wequi f'_* g'^*: \SH(Y_\et^\hyp) \to \SH(X_\et'^\hyp). \]
\end{lemma}
\begin{proof}
Clear by existence of left adjoints to étale pullback, stable under base change.
\end{proof}

\begin{lemma} \label{lemm:etale-f*-cocont}
Let $f: X \to Y$ be a qcqs morphism with $X, Y$ of uniformly bounded $p$-étale
cohomological dimension.
Then $f_*: \SH(X_\et^\hyp)_p^\comp \to \SH(Y_\et^\hyp)_p^\comp$ preserves
colimits and has finite $p$-cohomological
dimension: there exists $N$ such that $f_*(\SH(X_\et^\hyp)_{\ge i}/p) \subset
\SH(Y_\et^\hyp)_{\ge i-N}$, for all $i$.
\end{lemma}
\begin{proof}
Since $f$ is qcqs, $f^* = X \times_Y \bullet$ preserves qcqs schemes. It follows
now from Lemma \ref{lemm:uniformly-bounded-compact-generation} that
$f^*: \SH(Y_\et^\hyp)_p^\comp \to \SH(X_\et^\hyp)_p^\comp$ preserves compact generators.
Consequently $f_*$ preserves colimits.

Let $X$ be of uniformly bounded $p$-étale cohomological dimension $\le N$.
Let $A \in Y_\et$ be qcqs and $E \in \SH(X_\et^\hyp)_{\ge 0}$. Then $\map(A, f_*
E/p) \wequi \map(f^*A, E/p) \in \SH_{\ge -N}$, by Lemma
\ref{lemm:SH-compact-objects}(2) (and
Lemma \ref{lemm:uniformly-bounded-compact-generation}).
Since $A$ was arbitrary, this implies that
$f_*E/p \in \SH(Y_\et^\hyp)_{\ge -N}$. This concludes the proof.
\end{proof}

\begin{corollary}[homotopy invariance] \label{cor:htpy-inv}
Let $X$ be a scheme, $1/p \in X$ and suppose that there is an étale cover
$\{X_\alpha \to X\}_\alpha$ such that for each $\alpha$, both $X_\alpha$ and $\A^1
\times X_\alpha$ are of uniformly bounded $p$-étale cohomological dimension.
Then $q^*: \SH(X_\et^\hyp)_p^\comp \to \SH(\A^1 \times X_\et^\hyp)_p^\comp$ is
fully faithful.
\end{corollary}
\begin{proof}
Let $E \in \SH(X_\et^\hyp)$. We wish to prove that $E \to q_*q^*E$ is a $p$-equivalence, or
equivalently an equivalence mod $p$. We may thus assume that $E$ has $p^2$-torsion
homotopy sheaves. By Lemma
\ref{lemm:etale-base-change} we may assume that $X$ and $\A^1_X$ are of
uniformly bounded $p$-étale cohomological dimension. If $E \in
\SH(X_\et^\hyp)^\heart$, the result is \cite[Corollaire XV.2.2]{sga4}. The result for all
bounded above spectra follows by taking colimits via Lemma
\ref{lemm:etale-f*-cocont}. The general case follows from Lemma
\ref{lemm:coh-dimension-trick}.
\end{proof}

\begin{corollary}[proper base change] \label{cor:proper-base-change}
Consider a cartesian square
\begin{equation*}
\begin{CD}
Y' @>g'>> Y   \\
@Vf'VV  @VfVV \\
X' @>g>> X
\end{CD}
\end{equation*}
with $f$ proper.
Assume that there is an étale cover $\{X_\alpha \to X\}_\alpha$ such that each
$X_\alpha, X'_\alpha, Y_\alpha, Y'_\alpha$ is of uniformly bounded $p$-étale
cohomological dimension. (Here $Y_\alpha := Y \times_X X_\alpha$, and so on.)
Then $g^*f_* \wequi f'_* g'^*:
\SH(Y_\et^\hyp)_p^\comp \to \SH(X_\et'^\hyp)_p^\comp$.
\end{corollary}
\begin{proof}
Let $E \in \SH(Y)$. We need to prove that $g^*f_* E \to f'_* g'^* E$ is a
$p$-equivalence, i.e. an equivalence mod $p$. We may thus replace $E$ by $E/p$
and assume that $E$ has $p^2$-torsion homotopy sheaves. By Lemma
\ref{lemm:etale-base-change} we may assume that $X, X', Y, Y'$ are of
uniformly bounded $p$-étale cohomological dimension. If $E \in \SH(Y_\et^\hyp)^\heart$, the
result follows from \cite[Theorem XII.5.1]{sga4}. The functors $f_*, f'_*$
preserve colimits by Lemma \ref{lemm:etale-f*-cocont}, so we get the result for
all bounded above spectra. Moreover by the same lemma, all our functors are of
finite cohomological dimension. Thus we are done by Lemma
\ref{lemm:coh-dimension-trick}.
\end{proof}

We now come to the analog of Corollary \ref{cor:htpy-inv} for $\sigma$. Thus let
$S$ be a scheme and denote by $q: (\A^1 \setminus 0)_S \to S$ the canonical map.
Denote by $i: S \to (\A^1 \setminus 0)_S$ the inclusion at the point $1$.
For $E \in \SH(S_\et^\hyp)$ we consider the map $q_* q^* E \to q_*
(i_* i^*) q^* E \wequi E$, where the first map is the unit of adjunction and the
equivalence just comes from $qi = \id$. We denote by $E^\Gm$ the fiber of $q_*
q^* E \to E$.
\NB{The map $q_* q^* E \to E$ splits the unit $E \to q_* q^* E$, so we in fact have
$q_* q^* E \wequi E \vee E^\Gm$.}
Now let $E \in \SH(S_\et^\hyp)_p^\comp$. Consider the morphism $\tilde \sigma_E:
\imap(\Sptw[1],
E) \to q_*q^* E$ constructed as follows. Since $\Sptw$ is invertible, we have
$\imap(\Sptw[1], E) \wequi D(\Sptw[1]) \wedge E$. Now by adjunction we need to
construct a map $q^*(D(\Sptw[1]) \wedge E) \wequi D(\Sptw[1]) \wedge q^*(E) \to
q^*(E)$\footnote{Recall that we denote by $D(\ph)$ the passage to strong
duals.}, or equivalently a map $q^* E \to \Sptw[1] \wedge q^*(E)$. We take
$\id_{q^*E} \wedge \sigma$, where $\sigma: \1 \to \Sptw[1]$ is the map of
Definition \ref{def:sigma}. Note further that the composite
$\imap(\Sptw[1], E) \to q_*q^*E \to E \wequi q_* i_* i^* q^* E$ is trivial(ized):
this follows from the fact that the $\Zptw$-torsor $i^* C$ is trivial (see
Proposition \ref{prop:sigma}). Consequently $\tilde\sigma_E$ factors through
$E^\Gm$, yielding finally $\sigma_E: \imap(\Sptw[1], E) \to E^\Gm$.

\begin{proposition}[$\sigma$-locality] \label{prop:sigma-locality}
Let $X$ be a scheme, $1/p \in X$ and suppose that there is an étale cover $\{X_\alpha \to
X\}_\alpha$ such that each $X_\alpha$ and $(\A^1
\setminus 0) \times X_\alpha$ are of uniformly bounded $p$-étale cohomological dimension. Then
for every $E \in \SH(X_\et^\hyp)_p^\comp$, the map $\sigma_E: \imap(\Sptw[1], E)
\to E^\Gm$ is an equivalence.
\end{proposition}
\begin{proof}
We are trying to prove that a certain map is a $p$-equivalence. Since all
functors involved are stable, we may replace $E$ by $E/p$; hence we assume that
$E$ is $p^2$-torsion and need to prove that the appropriate map is a plain
equivalence. Using Lemma \ref{lemm:etale-base-change}, we may assume that
$X, \A^1_X$ are of uniformly bounded $p$-étale cohomological dimension. Recall that
tensoring with an invertible object is an equivalence, so preserves colimits;
hence $\map(\Sptw[1], \bullet)$ preserves colimits.
Note also that $D(\Sptw) \in \SH(X_\et^\hyp)_{\ge 0}$, since this
object is pro-étale locally equivalent to $\1$. Using Lemmas
\ref{lemm:etale-f*-cocont} and \ref{lemm:uniformly-bounded-compact-generation}
we may apply Lemma \ref{lemm:coh-dimension-trick}. Consequently we may assume
that $E$ is bounded above, which using cocontinuity we immediately
reduce to $E \in \SH(X_\et^\hyp)^\heart$. We may further assume that $E$
corresponds to a sheaf of $\Z/p$-vector spaces.

We have defined $E^\Gm$ as a summand of $q_*q^* E$, where $q: (\A^1 \setminus
0)_S \to S$ is the projection. Since $\P^1$ is covered by two copies of $\A^1$
with intersection $\A^1 \setminus 0$, using homotopy invariance for étale
cohomology \cite[Corollaire XV.2.2]{sga4}, we find that $E^\Gm$ is also a
summand of $r_*r^* E[-1]$, where $r: \P^1_S \to S$ is the projection. From proper
base change \cite[Theorem XII.5.1]{sga4} we deduce that $E^\Gm$ is stable under
arbitrary base change (for $E \in \SH(X_\et^\hyp)^\heart$). Of course
$\imap(\Sptw[1], E) \wequi D(\Sptw[1]) \wedge E$ is also stable under base
change. Using that geometric
points serve as stalks for the étale topology, we reduce to the case where $X =
Spec(k)$, $k$ a separably closed field. In this case $\Shv(X_\et^\hyp)$ is just
the topos of spaces, and so in particular $E$ corresponds to just a
$\Z/p$-vector space. Using that étale cohomology commutes with filtered
colimits, we reduce to the case $E = H\Z/p$. In this case $E^\Gm \wequi \H^1(\A^1
\setminus 0, \Z/p)[-1] \wequi H\Z/p[-1]$ \cite[Proposition
VII.1.1(ii)]{sga5}. Since also $\Sptw \wequi \1$ we find
that similarly $\imap(\Sptw[1], H\Z/p) \wequi H\Z/p[-1]$. Thus the map $\sigma_E$
we are trying to show is an equivalence corresponds simply to a map $\sigma:
\Z/p \to \Z/p$, which is an isomorphism if and only if it is non-zero. In fact
$\H^1_\et(\A^1 \setminus 0, \Z/p)$ classifies $\Z/p$-torsors, $\sigma$
corresponds to such a torsor, and we need to show this torsor is non-zero.
I claim that that $\sigma$ corresponds to $C_1$ (from the end of Section \ref{sec:proet}),
which is clearly non-trivial.

It thus remains to prove the claim. It suffices to show that the map $\sigma_1
\wedge H\Z/p: H\Z/p \to \Sptw[1] \wedge H\Z/p \wequi H\mu_p[1] \in \SH((\A^1 \setminus
0)_{X,\et}^\hyp)_p^\comp$ classifies $C_1$; here $\sigma_1: \1 \to \Sptw$ is the
map from the end of Section \ref{sec:proet}.
This claim is stable under base
change, so it suffices to prove this for $X = Spec(\Z[1/p])$, and hence by fully
faithfulness of $\nu^*$ we may prove it in $X_\proet$ instead. Recall that in
this context the
map $\sigma_1$ is given by $\Sigma^\infty \sigma_0$, where $\sigma_0: S^1 \to
K(\1, \Zptw) \wequi \Omega^\infty H\Zptw[1]$ classifies the torsor $C$. We
obtain the diagram
\[ \1 \xrightarrow{\eta} H\Z \xrightarrow{\sigma_1 \wedge H\Z} \Sptw[1] \wedge H\Z \wequi
H\Z \wedge \Sigma^\infty \Omega^\infty H\Zptw[1] \xrightarrow{\epsilon}
H\Zptw[1], \]
where $\eta$ is the  unit map and $\epsilon$ the co-unit of adjunction.
By construction, the composite map is adjoint to $\sigma_0$. The map $\epsilon$ is a
$p$-equivalence, by Lemma \ref{lemm:plausibility}. The claim follows.
\end{proof}

\begin{remark}[Clausen] \label{rmk:clausen}
Taking $E = \1$, we learn in particular that $q_*(\1) \wequi \1 \vee C$, where
$C \wequi \nSptw[-1]$. A sufficiently well-developed form of the six functors
formalism for étale cohomology with spectral coefficients should allow one to
prove ab initio that $C$ is invertible, circumventing our somewhat awkward
construction using the pro-étale topology.
\end{remark}

\section{The motivic category $\SH_\et(\bullet)$}
\label{sec:motivic-category}

Recall that a \emph{pre-motivic category} is a functor $\scr C: \Sch^\op \to
\Cat_\infty$, satisfying certain properties. Chiefly among them: each $\scr
C(X)$ is presentable, for each $f: X \to Y$ the functor $f^*: \scr C(Y) \to \scr
C(X)$ has a right adjoint $f_*$. If $f$ is smooth, there is a left adjoint
$f_\#$. The smooth base change formula holds. Typically one requires all $\scr
C(X)$ to be presentably symmetric monoidal and all $f^*$ to be symmetric
monoidal functors. Then the smooth projection formula is required to hold. One
then often asks for $\A^1$-invariance ($\A^1 \wequi *$) and $\P^1$-stability
($\P^1$ is an invertible object in the symmetric monoidal structure). Usually
each $\scr C(X)$ is also required to be stable; in this situation one may ask
that $\scr C$ should satisfy \emph{localization}: any decomposition $Z, U
\subset X$ into an open subset and closed complement should induce a
recollement. See \cite{triangulated-mixed-motives} for a careful statement. If
this holds, many further properties follow, and one says that $\scr C$ satisfies
the full six functors formalism.

The assignment $S \mapsto \SH(\Sm_{S,\et}^\hyp)$ defines a premotivic, stable
presentably symmetric monoidal category not satisfying any of the further assumptions. We
let $\SH^{S^1}_\et(S)$ be the $\A^1$-localization of $\SH(\Sm_{S,\et}^\hyp)$ and
$\SH_\et(S)$ the $\P^1$-stabilization. We recall the following fundamental
result.

\begin{theorem}[Ayoub] \label{thm:6-functors}
The premotivic categories $\SH_\et^{S^1}(\bullet), \SH_\et(\bullet)$ satisfy
localization. Hence $\SH_\et(\bullet)$ satisfies the full six functors formalism.
\end{theorem}
\begin{proof}
The localization axiom is verified in \cite[Corollaire 4.5.47]{ayoub2007six}.
What is implicit here is that a topology $\tau$ has been fixed, which is allowed
to be the étale topology: see the beginning of Section 4.5 in the reference.
Localization together with the remaining standard properties implies the full
six functors formalism; see Chapter 1 of the reference.
\end{proof}

We also wish to treat the \emph{continuity} axiom: usually this asks that for certain
pro-schemes $X = \lim_i X_i$ and any $E \in \scr C(X_0)$ we have $[\1,
E_X]_{\scr C(X)}
\wequi \colim_i [\1, E_{X_i}]_{\scr C(X_i)}$.
We shall say that \emph{$p$-continuity} holds if \[ [\1,
E_X/p]_{\scr C(X)} \wequi \colim_i [\1, E_{X_i}/p]_{\scr C(X_i)}. \]

We begin with the following
abstract result. It is a spectral analog of a considerable weakening of
\cite[Lemma 1.1.12]{cisinski2013etale}.

\begin{lemma} \label{lemm:pullback-sheaves}
Let $I$ be an essentially small filtering category and $(\scr C_i)_{i \in I}$ a
system of sites with colimit $\scr C$. Let $\scr X_i = \Shv(\scr C_i)^\hyp$,
$\scr X = \Shv(\scr C)^\hyp$. Suppose given for each $i$ a generating family
$\scr G_i \subset \scr X_i$. Write
$f_i^*: \scr X_i \to \scr X$ for the pullback, and $\scr G$ for the canonical
generating family of $\scr X$. Assume the following:
\begin{enumerate}
\item For each $i \in I$, each $X \in \scr G_i$ is coherent. Each $X \in \scr G$
  is coherent.
\item For each $\alpha: i \to j \in I$, the functor $\alpha^*$ has a left
  adjoint $f_\#$ preserving coverings.
\item For each $\alpha: i \to j \in I$, $(i \to j)^* \scr G_j \subset \scr G_i$.
  For each $i$, $f_i^*(\scr G_i) \subset \scr G$.
\end{enumerate}
Let $X \in \scr G$, $X = \lim_i X_i$ for some family of objects $\{X_i \in \scr
G_i\}$. Let $F \in \SH(\scr X_0)$. Assume that one of the following
conditions holds:
\begin{enumerate}[(a)]
\item $F \in \SH(\scr X_0)_{\le N}$ for some $N$.
\item $\SH(\scr X), \SH(\scr X_i)$ are ($p$-)postnikov complete for all $i$
  and there is $N$ such that $X$ and each of the $X_i$ has ($p$-)cohomological
  dimension $\le N$.
\end{enumerate}
Then\NB{similar statement for sheaves of spaces instead of spectra should also
hold}
\[ (f_0^* F')(X) \wequi \colim_{i \to 0} F'((i \to 0)_\# X_i). \]
Here $F'=F$ in case (a) and $F'=F/p$ in case (b).
\end{lemma}
\begin{proof}
We will put $f = f_0$, etc

(a) Consider $f^*_{pre}: \SH(\PSh(\scr C_0)) \to \SH(\PSh(\scr
C))$. It suffices to show that $f^*_{pre}$
preserves $N$-truncated (whence automatically hypercomplete) sheaves.
Since $F$ is $N$-truncated so is $f^*_{pre} F$ and we need to show it is a sheaf
(automatically hypercomplete). By the coherence assumption, this happens if (and
only if) $f^*_{pre} F$ (1) takes finite coproducts to products, and (2) satisfies
descent for morphisms of the form $f_j^*(Y_j \to X_j)$, where $Y_j \to X_j \in
\scr C_j$ is a covering. Since finite limits commute with
filtered colimits in spaces, condition (1) is clear. For condition (2), we are dealing
with a totalization (i.e. limit of a cosimplicial diagram) instead of a finite limit.
However, the homotopy groups of a
totalizations of an $N$-truncated diagram are determined by finite limits
\cite[Proposition 1.2.4.5(5)]{lurie-ha}, and hence do commute with filtered
colimits.

(b) Note that \[ [X, f^* F'] \wequi [X, (f^* F')_{\le N}] \wequi
[X, f^* (F'_{\le N})] \wequi \colim_i [X_i, F'_{\le N}] \wequi \colim_i [X_i, F'],
\] using both case (a) and Lemma \ref{lemm:SH-compact-objects}(1). This was to
be proved.
\end{proof}

In order to apply the above lemma, we need some preparations.

\begin{lemma}\label{lemm:easy-conservativity}
Let $\{f_\alpha: S_\alpha \to S\}_\alpha$ be an étale cover. The functors
$\{f_\alpha^*\}_\alpha$ form a conservative collection for $\SH(\Sm_{\bullet,
\et}^\hyp), \SH^{S^1}_\et(\bullet)$ and $\SH_\et(\bullet)$.
\end{lemma}
\begin{proof}
Since Čech nerves of étale covers have been inverted, the functors $\{f_{\alpha\#}\}_\alpha$
have jointly dense image. The result follows.
\end{proof}

Let $X \in \Sm_S$. We have a continuous map of sites $e_X: X_\et^\hyp \to
\Sm_{S,\et}^\hyp$ inducing $e_X^*: \PSh(X_\et) \adj \PSh(\Sm_{S,\et}^\hyp):
e_{X*}$.\NB{$e_X^*$ does not preserve finite products!}

\begin{lemma}\label{lemm:eX*-preserves-local-equiv}
The functor $e_{X*}$ preserves étale-hyperlocal equivalences (that is, those
maps of presheaves of spaces or spectra which become equivalences when taking
the associated hypercomplete sheaf).\NB{also true for
ordinary étale equivalences}
\end{lemma}
\begin{proof}
Since étale-hyperlocal equivalences can be tested on stalks, and any stalk of
$X_\et$ is also a stalk of $\Sm_{S,\et}$, the result follows.
\end{proof}

\begin{corollary} \label{cor:eX*-preserves-trunc}
The functor $e_{X*}: \Shv(\Sm_{S,\et}^\hyp) \to \Shv(X_\et^\hyp)$ preserves
colimits and truncations. Similarly for $e_{X*}: \SH(\Sm_{S,\et}^\hyp) \to
\SH(X_\et^\hyp)$.
\end{corollary}
\begin{proof}
Colimits of sheaves are computed as colimit of presheaves (i.e. sectionwise)
followed by
hypersheafification. Since the presheaf version of $e_{X*}$ preserves colimits,
Lemma \ref{lemm:eX*-preserves-local-equiv} implies that the sheaf version also
does. The claim about truncations follows by the same argument.
\end{proof}

\begin{lemma} \label{lemm:big-etale-postnikov}
Suppose that every smooth $S$-scheme is étale locally of uniformly bounded
($p$-)étale cohomological
dimension. Then ($p$-)postnikov towers converge in $\SH(\Sm_{S, \et}^\hyp)$, and all
the objects $\Sigma^\infty X_+$ for $X \in \Sm_S$ with $\cd(X) < \infty$
are compact.
If $\cd_p X < \infty$ then $\Sigma^\infty X_+/p$ is compact both in $\SH(\Sm_{S,
\et}^\hyp)$ and in $\SH(\Sm_{S, \et}^\hyp)_p^\comp$.
\end{lemma}
\begin{proof}
The collection of functors $e_{X*}$ for various $X \in \Sm_S$ is clearly
conservative and commutes with limits, and also with truncations by Corollary
\ref{cor:eX*-preserves-trunc}. It is hence enough to show that
$\SH(X_\et^\hyp)$ ($p$-)postnikov complete, objects of finite cohomological
dimension are compact in $\SH(X_\et^\hyp)$, and that $e_{X*}$ preserves filtered colimits.
The first two statements are proved in Lemma \ref{lemm:uniformly-bounded-compact-generation},
and the last one is Corollary \ref{cor:eX*-preserves-trunc}.
\end{proof}

\begin{corollary} \label{cor:SH-et-compactly-gen}
Under the assumptions of Lemma \ref{lemm:big-etale-postnikov}, $\SH_\et(S)$
(respectively $\SH_\et(S)_p^\comp$) is
compactly generated by $\Sigma^\infty X_+ \wedge \Gmp{n}$ (respectively
$\Sigma^\infty X_+ \wedge \Gmp{n}/p$) for $n \in \Z$ and $X
\in \Sm_S$ with $\cd(X) < \infty$ (respectively $\cd_p(X) < \infty$).
A similar statement holds for $\SH^{S^1}$.
\end{corollary}
\begin{proof}
Compact generators are preserved under stabilization with respect to a symmetric
object \cite[Proof of Lemma 4.1]{bachmann-norms} and
$\A^1$-localization (or more generally any localization at a family of maps
between compact objects). The same proof works for $\SH^{S^1}$.
\end{proof}

\begin{definition} \label{def:etale-finite}
We say that a scheme $S$ is \emph{($p$-)étale finite} if every finite type $S$-scheme is of
uniformly bounded ($p$-)étale cohomological dimension. We call $S$ \emph{locally
($p$-)étale finite} if there exists an étale cover $\{S_\alpha \to S\}_\alpha$ with each
$S_\alpha$ ($p$-)étale finite.
\end{definition}

\begin{example} \label{ex:etale-finite}
$S$ is ($p$-)étale finite whenever Corollary
\ref{cor:uniformly-finite-etale-dimension-example} applies. In particular all
the schemes from Example \ref{ex:uniformly-bounded-dim} are (locally) étale
finite.
\end{example}

\begin{proposition} \label{prop:continuity}
Let $S_0$ be qcqs and ($p$-)étale finite, and $\{S_i\}_i$
a pro-scheme with each $S_i$ étale and affine over $S_0$. Then
$\SH(\Sm_{\bullet, \et}^\hyp), \SH^{S^1}_\et(\bullet)$ and $\SH_\et(\bullet)$
satisfy ($p$-)continuity for the pro-system $\{S_i\}_i$.
\end{proposition}
\begin{proof}
By Lemma \ref{lemm:uniform-dim-permanence}, $S$ is ($p$-)étale finite.
We wish to apply Lemma \ref{lemm:pullback-sheaves}(b) with
$\scr C_i = \Sm_{S_i, \et}$, $X_0$ some smooth, quasi-separated $S_0$-scheme.
We can do this by Lemma
\ref{lemm:big-etale-postnikov} (which says that the required postnikov towers
converge) and the definition of ($p$-)étale finiteness (which ensures that the
$X_i$ have bounded ($p$-)étale cohomological dimension). We conclude that if $f: S \to
S_0$ is the projection, then for $E \in \SH(\Sm_{S_0,\et}^\hyp)$ and $X$ smooth
and quasi-separated, we have $(f^* E')(f^* X) \wequi \colim_i E'(X_i)$. (Here
$E' = E$ or $E'=E/p$, as appropriate.) In particular,
$\SH(\Sm_{\bullet, \et}^\hyp)$ satisfies ($p$-)continuity for this system (take
$X_0 = S_0$).

Since smooth quasi-separated schemes generate our categories, we conclude also that
$f^*: \SH(\Sm_{S_0, \et}^\hyp) \to \SH(\Sm_{S, \et}^\hyp)$ preserves
$\A^1$-local objects. This implies continuity for $\SH^{S^1}_\et$. Since
(filtered) colimits of spectra commute with finite limits,
the functor $f^*: \SH^{S^1}_\et(S_0) \to \SH^{S^1}_\et(S)$ preserves
$\Gm$-$\Omega$-spectra. This implies continuity for $\SH_\et(\bullet)$.
\end{proof}

\begin{corollary} \label{cor:intermediate-conservativity}
Let $S$ be locally ($p$-)étale finite. Then the
family of functors $i_{S_{\bar x}}^*: F(S) \to F(S_{\bar x})$ is conservative,
where $F$ is one of $\SH(\Sm_{\bullet, \et}^\hyp), \SH^{S^1}_\et(\bullet)$ or
$\SH_\et(\bullet)$ (respectively their $p$-completions),
$\bar x$ runs through geometric points of $S$ and $S_{\bar x}$ denotes the strict henselization.
\end{corollary}
\begin{proof}
Let $\{S_\alpha\}$ be an étale covering by schemes which are qcqs and ($p$-)étale
finite. Any geometric point of some $S_\alpha$ is also a
geometric point of $S$, hence by Lemma \ref{lemm:easy-conservativity} we may
assume $S$ qcqs and ($p$-)étale finite.  Let $\{Y_i\}$ be a pro-system
representing a geometric point $\bar y$ of $Y \in \Sm_S$ and $E \in F(S)$ with
$i_{S_{\bar x}}^*(E)
\wequi 0$ for all $\bar x$. Put $E' = E/p$ if we are working at a prime $p$, or
$E' = E$ else. We shall show that (*) $\colim_i E'(Y_i) \wequi 0$. Since
this holds for all stalks, we conclude that $0 \wequi \Omega^\infty E' \in
\PSh(\Sm_S)$. Since this also applies to all shifts (and twists if $F =
\SH_\et$) of $E'$, the result follows.

It remains to prove (*). Let $Y = \lim Y_i$. Let $\bar x$ be the geometric point
of $S$ under $\bar y$. Then $Y \to S$ factors through $S_{\bar x} \to S$, hence
$E'_Y \wequi 0$. Since $Y$ is a henselization, we may assume that $Y_0$ (and in
fact each $Y_i$) is qcqs. Since also $Y_0$ (and in fact each $Y_i$) is étale finite,
Proposition \ref{prop:continuity} applies: we have ($p$-)continity for $\{Y_i\}_i$. Hence $0
\wequi E'_Y(\1) \wequi \colim_i E'(Y_i)$. The result follows.
\end{proof}

\begin{corollary} \label{cor:conservativity}
Let $S$ be locally étale finite and of finite dimension. Then the
family of functors $i_{\bar x}^*: \SH_\et(S) \to \SH_\et(\bar{x})$ is
conservative, where $\bar x$ runs through geometric points of $S$. The same is
true for $\SH^{S^1}$.
If instead $S$ is locally $p$-étale finite of finite dimension, then the same
conclusions hold for the $p$-complete categories.
\end{corollary}
\begin{proof}
We prove this by induction on the dimension of $S$. Let $E=E' \in \SH_\et(S)$
(respectively $E \in \SH_\et(S)_p^\comp$, $E'=E/p$), with
$E_{\bar x} \wequi 0$ for all geometric points $\bar x$ of $S$. Let $\bar y$ be a geometric
point of $S$; by Lemma \ref{lemm:uniform-dim-permanence} $S_{\bar y}$ is ($p$-)étale
finite. The result thus holds for
$U_{\bar y} := S_{\bar y} \setminus \bar y$, by induction. Hence $E'_{\bar y} \wequi 0$ (by
assumption) and $E'_{U_{\bar y}} \wequi 0$. By localization (Theorem
\ref{thm:6-functors}), $E'_{S_{\bar y}} \wequi 0$. Since $\bar y$ was arbitrary,
the result follows from Corollary \ref{cor:intermediate-conservativity}. The
proof for $\SH^{S^1}$ is the same.
\end{proof}

\section{Main results}
\label{sec:main}

\begin{lemma} \label{lemm:eX-ff}
Let $S$ be a scheme. The functor $e^*: \Shv(S_\et^\hyp) \to
\Shv(\Sm_{S,\et}^\hyp)$ is fully faithful. Similarly for $\SH(S_\et^\hyp) \to
\SH(\Sm_{S,\et}^\hyp)$
\end{lemma}
\begin{proof}
We need to prove that $e_* e^* \wequi \id$. Since stabilization is natural
for functors preserving finite limits, the claim for spectra
immediately follows from the claim for sheaves. If $X \in S_\et$, then $e_* e^*
X \wequi X$, since the étale topology is sub-canonical. Since representable
sheaves generate  $\Shv(S_\et^\hyp)$ under colimits, it suffices to show that
$e_*$ preserves colimits. This is Corollary \ref{cor:eX*-preserves-trunc}.
\end{proof}

Let $1/p \in S$. We construct a map $\sigma: \Gm \to \Sptw[1] := e^*(\Sptw[1])
\in \SH(\Sm_{S,\et}^\hyp)_p^\comp$
as follows. Since $\Gm$ is the cofiber of $i_0: S^0 \to (\A^1 \setminus 0)_+$,
constructing $\sigma$ is the same as constructing a map $\1 \to p^* e^*
\Sptw[1]$, where $p: (\A^1 \setminus 0)_S \to S$ is the canonical map, such that
the composite $\1 \to \Sigma^\infty (\A^1 \setminus 0)_{S+} \to \Sptw[1]$ is
trivial(ized). Since $p^* e^* \wequi e^* p^*$, for this we can use $e^*$ of the map
constructed at the end of Section \ref{sec:proet}.

We denote by $\sigma$ also
its image in $\SH^{S^1}_\et(S)_p^\comp$ and $\SH_\et(S)_p^\comp$. We denote by
$\SH^{S^1}_{\et, \sigma}(S)_p^\comp$ the monoidal localisation of
$\SH^{S^1}_\et(S)_p^\comp$ at the map $\sigma$; in other words this is the localization
at all the maps $\id_{\Sigma^\infty X_+} \wedge \sigma$ for $X \in \Sm_S$.

\begin{corollary} \label{cor:e-ff-sigma-local}
Let $S$ be locally $p$-étale finite, and assume that $1/p \in S$. Then the
canonical functor $\SH(S_\et^\hyp)_p^\comp \to \SH^{S^1}_{\et,\sigma}(S)_p^\comp$ is fully
faithful.
\end{corollary}
\begin{proof}
By Lemma \ref{lemm:eX-ff}, it suffices to prove that if $E \in
\SH(S_\et^\hyp)_p^\comp$, then $e^*(E) \in \SH(\Sm_{S,\et}^\hyp)_p^\comp$ is
$\A^1$-local and $\sigma$-local. In other words for $X
\in \Sm_S$, the maps $e^*(E)(X) \to e^*(E)(\A^1_X)$ and $e^*(E)(X_+ \wedge \Sptw[1]) \to
e^*(E)(X_+ \wedge \Gm)$ are equivalences. By the projection formula (and since
$\Sptw[1]$ is stable under base change), we may replace $S$ by $X$ and so
assume that $S=X$. Note that $e^*(E)(\A^1) \wequi \map(\1, q^* e^* E) \wequi
\map(\1, q^* E) \wequi \map(\1, q_*q^* E)$, where $q: \A^1_S \to S$ is the projection.
Similarly $e^*(E)(\Gm) \wequi \map(\1, E^\Gm)$ and $e^*(E)(\Sptw[1]) \wequi \map(\Sptw[1],
e^* E) \wequi \map(\1, \imap(\Sptw[1], E))$.
Thus this is the content of Corollary \ref{cor:htpy-inv} and Proposition
\ref{prop:sigma-locality}.
\end{proof}

Let $k$ be a field, say of finite étale cohomological dimension.
Recall the \emph{étale motive} functor $M: \SH_\et(k)_p^\comp \to \DM_\et(k,
\Z/p)$: e.g. in \cite{bachmann-hurewicz} the functor $\SH(k) \to \DM(k) \to
\DM(k, \Z/p)$ is defined; upon localization and $p$-completion this induces
$\SH_\et(k)_p^\comp \to \DM_\et(k, \Z/p)_p^\comp \wequi \DM_\et(k, \Z/p)$ since
$\DM_\et(k, \Z/p)$ is already $p$-complete (all objects being $p$-torsion).

\begin{lemma} \label{lemm:sigma-realization}
The map $M(\sigma): \1(1)[1] \to M\Sptw[1] \in \DM_\et(k, \Z/p)$ is an equivalence.
\end{lemma}
\begin{proof}
Lemma \ref{lemm:plausibility} together with \cite[Theorem
4.5.2]{cisinski2013etale} implies that $M\Sptw[1] \wequi \mu_p[1]$, and hence
$M(\sigma)$ classifies a $\mu_p$-torsor on $\A^1 \setminus 0$ (which is trivial
over $1$). As in the proof of Proposition \ref{prop:sigma-locality},
tracing through the definitions, we find that this torsor
is essentially $C_1$, so in particular non-trivial. Since $\1(1) \wequi \mu_p
\in \DM_\et(k, \Z/p)$
\cite[Proposition 3.2.2]{cisinski2013etale}, we find that $M(\sigma)$ is an
equivalence (since it is non-zero).
\end{proof}

\begin{lemma} \label{lemm:tau}
Let $k$ be a separably closed field, $1/p \in k$. There is a map $\tau: \1[1]
\to \Gm \in \SH^{S^1}_\et(k)_p^\comp$ such that $M(\tau): \1[1] \to M(\Gm) \wequi \1[1] \in
\DM_\et(k, \Z/p)$ is an equivalence.
\end{lemma}
\begin{proof}
We use the following facts, which hold for any object $E$ in a presentable
stable $\infty$-category\todo{reference?}\footnote{Points (1) and (3) follow
from the same statement in $\SH$, using that $E/p^n = \1/p^n \wedge E$. The
distinguished triangle in (3) can be constructed using the octahedral axiom. To
see that $E \to E_p^\comp$ is a $p$-equivalence, note that $E_p^\comp \wequi
\imap(\1/p^\infty[-1], E)$ where $\1/p^\infty = \colim_n \1/p^n$. It is thus enough
to show that $\1/p^\infty/p \wequi \1/p$, which again follows from the same
statement in $\SH$. For a treatment in the case of $\SH$, see e.g.
\cite[Section 2]{bousfield1979localization}.}
\begin{enumerate}
\item There is an inverse system $E \to \dots \to E/p^3 \xrightarrow{r_2} E/p^2
  \xrightarrow{r_1} E/p$.
\item The induced map $E \to \lim_n E/p^n =: E_p^\comp$ is a $p$-equivalence.
\item For each $n$ there are distinguished triangles \[ E/p \to E/p^{n+1}
  \xrightarrow{r_n} E/p^n. \]
\end{enumerate}
The Milnor sequence  \cite[Proposition VI.2.15]{goerss2009simplicial}
then shows that there is for any other object $F$ a surjection $[F, E_p^\comp]
\to \lim_n [F, E/p^n]$.

We apply this to the category $\SH^{S^1}(k)$. Note that there is no étale
localization! Then we know that $[\1, \Gm]_{\SH^{S^1}(k)} = K_1^{MW}(k)$ and
$[\1[-1], \Gm]_{\SH^{S^1}(k)} = 0$ \cite[Corollary 6.43]{A1-alg-top}. Choose a
primitive $p$-th root of unity $\zeta \in k^\times$. Since $-1$ is a square in
$k$, we have $\lra{-1} = 1 \in GW(k)$ and hence $p[\zeta] = p_\epsilon [\zeta] =
[\zeta^p] = [1] = 0 \in K_1^{MW}(k)$ \cite[Lemma 3.14]{A1-alg-top}. Consequently
$[\zeta] \in ker([\1, \Gm] \xrightarrow{p} [\1, \Gm])$. Let $\tau_1 \in [\1[1],
\Gm/p]$ lift $[\zeta]$. We shall construct a sequence of elements $\{\tau_n \in
[\1[1], \Gm/p^n]\}_n$ with $r_{n}(\tau_{n+1}) = \tau_n$. This implies that the
$\{\tau_n\}_n$ define an element of $\lim_n [\1[1], \Gm/p^n]$ and hence lift to
an element $\tau \in [\1[1], \Gm_p^\comp]$. We construct the $\{\tau_n\}_n$
inductively. Suppose $\tau_n$ is constructed. Using (3), $\tau_{n+1}$ exists if
and only if the image of $\tau_n$ under the boundary map $[\1[1], \Gm/p^n] \to
[\1[1], \Gm/p[1]]$ is zero. It is thus enough to show that $[\1, \Gm/p] = 0$.
Since $[\1[-1], \Gm] = 0$ we have an exact sequence $[\1, \Gm] \xrightarrow{p}
[\1, \Gm] \to [\1, \Gm/p] \to 0$; it is hence enough to show that $K_1^{MW}(k)
\xrightarrow{p} K_1^{MW}(k)$ is surjective. Let $a_1, \dots, a_m \in k^\times$; then
$K_1^{MW}(k)$ is generated by elements of the form $\eta^{m-1}[a_1] \dots [a_m]$
\cite[Lemma 3.4]{A1-alg-top}.
As observed above,
$p[a_1] = [a_1^p]$; it is thus enough to show that $k^\times \xrightarrow{p}
k^\times$ is surjective. Since $k$ is separably closed of characteristic $\ne p$,
this is clear.

We define the map $\tau$ required for this lemma to be the étale locacalization
of the map $\tau$ we constructed. It remains to show that $M(\tau)$ is an
equivalence. Let $M': \SH^{S^1}(k) \to \DM(k, \Z/p)$ denote the ``associated
Nisnevich motive'' functor. Then $M'(\tau): \1[1] \to \1(1)[1] \in \DM(k, \Z/p)$
corresponds to a map $\tau': \1[1] \to \Gm \wedge H\Z/p \in \SH(k)$, where $H\Z/p$
denotes the motivic Eilenberg-MacLane spectrum. By construction, $\tau'$ is
homotopic to the composite $\1[1] \xrightarrow{\tau_1} \Gm/p \to \Gm/p \wedge H\Z \wequi \Gm
\wedge H\Z/p$. It follows that $\tau'$ corresponds precisely to
the map $\beta: \1 \to \1(1) \in \DM(k, \Z/p)$ as constructed for example in
\cite[bottom of p. 202]{hasemeyer2005motives}, which is well-known to induce an
equivalence after étale localization.
\end{proof}

\begin{theorem}
Let $S$ be a scheme, $1/p \in S$. Then $\sigma: \Gm \to \Sptw[1] \in
\SH_\et(S)_p^\comp$ is an equivalence.
\end{theorem}
\begin{proof}
It suffices to treat the case $S = Spec(\Z[1/p])$. By Example
\ref{ex:etale-finite}, $S$ is now locally étale finite.
By Corollary \ref{cor:conservativity} (and stability
of $\sigma$ under base change), we may thus reduce to $S=Spec(k)$, where $k$ is
a separably closed field with $1/p \in k$. Consider the composite $\sigma \tau:
\1[1] \to \Sptw(1)[1] \wequi \1[1] \in \SH^{S^1}_\et(k)_p^\comp$. Then $M(\sigma
\tau)$ is an equivalence, by Lemmas \ref{lemm:sigma-realization} and
\ref{lemm:tau}. Corollary \ref{cor:e-ff-sigma-local} implies that $M: [\1[1],
\1[1]]_{\SH^{S^1}_\et(k)_p^\comp} \to [\1[1], \1[1]]_{\DM(k, \Z/p)}$ is the
reduction map $\Z_p \to \Z/p$. It follows that $\sigma \tau$ is
an equivalence. Hence $\Gm[-1]$ splits off a copy of $\1$ in
$\SH^{S^1}_\et(k)_p^\comp$, and hence also in $\SH_\et(k)_p^\comp$. It now
follows from \cite[Lemma 30]{bachmann-real-etale} that $\sigma$ is a
$\Gm$-stable equivalence (and so is $\tau$). This concludes the proof.
\end{proof}

Consequently for any scheme $S$ the canonical functor
$\SH^{S^1}_\et(S)_p^\comp \to \SH_\et(S)_p^\comp$ factors through
$\SH^{S^1}_{\et,\sigma}(S)_p^\comp$. Since $\Gm$ is invertible in
$\SH^{S^1}_{\et,\sigma}(S)_p^\comp$, we find that the induced functor
$\SH^{S^1}_{\et,\sigma}(S)_p^\comp \to \SH_\et(S)_p^\comp$ is an
equivalence: inverting $\Gm$ commutes with localizations (like $p$-completion),
see e.g. \cite[Lemma 26]{bachmann-real-etale}, and inverting an already
invertible object does not do anything, as is clear from the universal property.

\begin{theorem} \label{thm:main}
Let $S$ be locally $p$-étale finite, and $1/p \in S$. Then the canonical functor
$e: \SH(S_\et^\hyp)_p^\comp \to \SH_\et(S)_p^\comp$ is an equivalence.
\end{theorem}
\begin{proof}
We use the argument from
\cite[Theorem 4.5.2]{cisinski2013etale}: the functor $e$ is fully faithful and
preserves colimits, by the above remarks and Corollary
\ref{cor:e-ff-sigma-local}. It hence identifies $\SH(S_\et^\hyp)_p^\comp$ with a
localizing subcategory of $\SH_\et(S)_p^\comp$. We need to show it is
essentially surjective. The category $\SH_\et(S)$ is generated by
objects of the form $q_*(\1 \wedge \Gmp{n}) \wequi q_*(\1) \wedge \Gmp{n}$,
where $q: T \to S$ is projective
\cite[Proposition 4.2.13]{triangulated-mixed-motives}; hence the same holds for
$\SH_\et(S)_p^\comp$. Since $\Gm \wequi
\Sptw[1] \in \SH_\et(X)_p^\comp$, and $\Sptw$ (and all its positive or negative
powers) are in the essential image of $e$, it thus suffices to show
that $e$ commutes with $q_*$. This is proved exactly as in \cite[Proposition
4.4.3]{cisinski2013etale}: it boils down to the fact that both sides satisfy
proper base change (see Corollary
\ref{cor:proper-base-change} and Theorem \ref{thm:6-functors}).
\end{proof}

\section{Applications}
\label{sec:applications}

\begin{theorem}
Let $1/p \in S$ and assume that $S$ is locally étale finite.
There is a symmetric monoidal ``étale realization''
functor $\SH_\et(S) \to \SH(S_\et^\hyp)_p^\comp$.
\end{theorem}
\begin{proof}
Take the composite of the symmetric monoidal localization $\SH_\et(S) \to
\SH_\et(S)_p^\comp$ with the symmetric monoidal inverse of the symmetric
monoidal equivalence $\SH(S_\et^\hyp)_p^\comp \to \SH_\et(S)_p^\comp$ from
Theorem \ref{thm:main}.
\end{proof}

We can also determine the endomorphisms of the unit in $\SH_\et(S)$ without
$p$-completion, at least away from primes non-invertible in $S$, and under some
more restrictive hypotheses on $S$.

\begin{theorem} \label{thm:app2}
Let $X$ be locally étale finite. Let $S$ be the set of primes not invertible on $X$.
Consider the functor $e: \SH(X_\et^\hyp)[1/S] \to \SH_\et(X)[1/S]$.
\begin{enumerate}
\item If $X$ is regular, noetherian and finite dimensional, then $e$ is fully
  faithful when restricted to the localizing subcategory generated by $\1$.
\item If $X$ is the spectrum of a field, then $e$ is fully faithful.
\end{enumerate}
\end{theorem}
\begin{proof}
Write $e: \SH(X_\et^\hyp) \adj \SH_\et(X): e_*$ for the adjunction. Under our
assumptions, $e$ preserves compact generators by Corollary
\ref{cor:SH-et-compactly-gen} and Lemma \ref{lemm:SH-compact-objects}. In
particular $e_*$ preserves $S$-localizations, rationalizations and so on.
We wish to prove that $\eta: \id \Rightarrow e_*e$ is an isomorphism on
$\SH(X_\et^\hyp)[1/S]$ (in case (2)) or the
localizing subcategory thereof generated by the unit (in case (1)).
Theorem \ref{thm:main} implies that $\eta/p$ is an
equivalence for $p$ invertible on $X$. It follows that the cofiber of $\eta$ is
uniquely $p$-divisible for all $p$ invertible on $X$, and also for all the other
$p$ by construction (i.e. since we are working in $\SH(X_\et^\hyp)[1/S]$). Thus
in order to show that $\eta$ is an equivalence it suffices to show that it is a
rational equivalence. We thus reduce to showing that $e_\Q: \SH(X_\et^\hyp)_\Q
\to \SH_\et(X)_\Q$ is fully faithful on an appropriate subcategory.

Recall that $\SH(X)_\Q \wequi D_{\A^1}(X, \Q) \wequi D_{\A^1}(X, \Q)^+
\times D_{\A^1}(X, \Q)^-$, for essentially any $X$
\cite[16.2.1.6]{triangulated-mixed-motives}. For $X$ noetherian,
finite dimensional and geometrically unibranch (e.g. regular \cite[Tag
0BQ3]{stacks-project}),
we have $D_{\A^1}(X, \Q)^+ \wequi \DM(X, \Q) \wequi \DM_\mathcyr{B}(X)$ \cite[Theorems 16.1.4 and
16.2.13]{triangulated-mixed-motives}. Moreover, all objects in $\DM(X, \Q)$
satisfy étale hyperdescent \cite[16.1.3]{triangulated-mixed-motives} (note that
in that reference, ``descent'' means ``hyperdescent'' \cite[Definition
3.2.5]{triangulated-mixed-motives}). Recall from Example
\ref{ex:uniformly-bounded-dim} that étale-locally on any scheme, $-1$ is a sum
of squares. It thus follows from \cite[Corollary
16.2.14]{triangulated-mixed-motives} that the identity endomorphism of the unit
object in $D_{\A^1}(X, \Q)^-$ vanishes étale-locally. This implies (via Lemma
\ref{lemm:easy-conservativity}) that the image of $\1_{D_{\A^1}(X, \Q)^-}$ in
$\SH_\et(X)_\Q$ is zero, and hence the étale hyperlocalization of $D_{\A^1}(X,
\Q)^-$ is zero. All in all we have obtained:
\[ \SH_\et(X)_\Q \wequi \DM(X, \Q) \wequi \DM_\et(X, \Q) \wequi
    \DM_\mathcyr{B}(X). \]

It remains to prove then that $e: D(X_\et^\hyp, \Q) \to \DM(X, \Q)$ is fully
faithful on an appropriate subcategory. In case (2), the left hand side is
compact-rigidly generated (by comparison with Galois cohomology). This implies that $e$
is fully faithful as soon as $e_* e(\1) \wequi \1$ \cite[Lemma
22]{bachmann-hurewicz}. Hence in either case we are reduced to proving
that $e_*(\1) \wequi \1$. The problem is local on $X$, so we may assume that $X$
is affine. We have $[\1, \1[i]]_{D(X_\et^\hyp, \Q)} \wequi
\H^i_\et(X, \Q)$, which $=0$ for $i \ne 0$ (see e.g. \cite[2.1]{DENINGER1988231}),
while $\H^0_\et(X, \Q) \wequi \H^0_\Zar(X, \Q)$ (e.g. by Lemma
\ref{lemm:constant-sheaves}). We need to show that $[\1,
\1[i]]_{\DM_\mathcyr{B}(X)}$ gives the same answer. Now we use that since $X$ is
regular we have $[\1, \1[i]]_{\DM_\mathcyr{B}(X)} = Gr^0_\gamma K_{-i}(X)_\Q$
\cite[Corollary 14.2.14]{triangulated-mixed-motives}. This is $= 0$ if $i > 0$
since $X$ is regular, and $=0$ for $i < 0$ by \cite[sentence before Proposition
IV.5.10]{weibel-k-book}. For $i=0$ we precisely get $\H^0(X, \Q)$ \cite[Theorem
II.4.10(4)]{weibel-k-book}. This concludes the proof.
\end{proof}

\begin{corollary} \label{corr:final}
With the notation and assumptions of Theorem \ref{thm:app2}, we have
\[ [\1[n], \1]_{\SH_\et(X)[1/S]} \wequi \mathbb{H}^{-n}_\et(X, \1[1/S]), \]
where the right hand side denotes étale hypercohomology with coefficients in the
(classical) sphere spectrum.
\end{corollary}
\begin{proof}
This is just a restatement of the theorem: $\mathbb{H}^{-n}_\et(X, \1[1/S]) \wequi [\1,
\1[n]]_{\SH(X_\et^\hyp)[1/S]}$, essentially by definition.
\end{proof}

\bibliographystyle{alpha}
\bibliography{bibliography}

\end{document}